\documentclass[12pt]{article}
\usepackage{tikz}
\usepackage{amsmath,amsfonts,amssymb,graphicx}
\usepackage{mathrsfs}
\usepackage[latin1]{inputenc}
\usepackage{multicol}
\usepackage[all]{xy}
\usepackage{amsthm}
\usepackage{tcolorbox}
\usepackage{color}
\usepackage{cancel}
\usepackage[misc]{ifsym}
\usepackage[bookmarksnumbered,colorlinks]{hyperref}
\usepackage{mathrsfs}
\usetikzlibrary{matrix}
\usepackage{epstopdf}
\usepackage{pstricks}
\usepackage{pst-plot}
\usepackage{pstricks-add}
\usepackage{epsf}
\usepackage{epic}
\usepackage{subfig}
\usepackage{fancyhdr}
\usepackage{fancybox}
\usepackage{boxedminipage}
\usepackage{shadow}
\usepackage{anysize}
\usepackage{extramarks}
\usepackage{fancybox}
\usepackage{multicol}
\usepackage{comment}
\usepackage{xcolor}
\usepackage{newtxmath}
\usepackage{tikz}
\usepackage{calc}
\usepackage{graphicx}
\usepackage{imakeidx}
\usepackage{diagbox}
\usepackage{caption}

\newcommand{\dst}{\displaystyle}
\newcommand{\mb}{\mathbb}

\newcommand{\wangle}{\widetilde{\measuredangle}}
\newcommand{\mangle}{\measuredangle}

\usepackage[top=3cm,bottom=3cm,hmargin=2.5cm]{geometry}

\newtheorem{theorem}{Theorem}%  meant for continuous numbers
\newtheorem{proposition}[theorem]{Proposition}
\newtheorem{lemma}[theorem]{Lemma}%

\newtheorem{example}{Example}%
\newtheorem{remark}{Remark}%

\newtheorem{definition}{Definition}%

\begin{document}

\title{Comparison theorems for Lorentzian length spaces with lower timelike curvature bounds}

\author{Waldemar Barrera, Luis Montes de Oca and Didier A. Solis}
\date{}
%\email{bvargas@correo.uady.mx}
%\equalcont{These authors contributed equally to this work.}

%\author[1]{\fnm{Luis} \sur{Montes de Oca}}\email{mauricio.montesdeoca@alumnos.uady.mx}
%\equalcont{These authors contributed equally to this work.}

%\author*[1]{\fnm{Didier A} \sur{Solis}}\email{didier.solis@correo.uady.mx}
%\equalcont{These authors contributed equally to this work.}

%\affil*[1]{\orgdiv{Facultad de Matem\'aticas}, \orgname{Universidad Aut\'onoma de Yucata\'an}, \orgaddress{\street{Anillo Perif\'erico 13615}, \city{M\'erida}, %\postcode{100190}, 
%\state{Yucat\'an}, \country{M\'exico}}}

\maketitle

\abstract{In this article we introduce a notion of normalized angle for Lorentzian pre-length spaces. This concept allows us to prove some equivalences to the definition of timelike curvature bounds from below for Lorentzian pre-length spaces. Specifically, we establish some comparison theorems known as the local Lorentzian version of the Toponogov theorem and the Alexandrov convexity property. Finally, as an application we obtain a first variation Formula for non-negatively curved globally hyperbolic Lorentzian length spaces.} 

\bigskip

\textbf{Keywords:} {Lorentzian length spaces, triangle comparison, spaces of bounded curvature, first variation.}

%%\pacs[JEL Classification]{D8, H51}

\textbf{MSC Classification:} {53C23, 53C50, 53C80.}

\section{Introduction}

The beginning of the XXI century has brought a lot of excitement and whole new perspectives both to Mathematical Relativity and Lorentzian geometry. In particular,  the first detection of gravitational waves \cite{LIGO} and the latest observations of black holes \cite{BH} have boosted the interest of exploring geometric tools that adapt well to non-smooth settings. 

In the context of Lorentzian geometry, the search for an axiomatic approach to the most remarkable aspects of relativity --such as causality--- can be traced back to the fundamental work of Kronheimer and Penrose \cite{KP}. In their approach, the main features of causality theory could be established from a small set of axioms rather than deduced from the smooth geometry structure of spacetime. Thanks to Penrose's insight, such idealizations have provided effective foundations to deal with situations where smoothness is not required, such as in the study of quantum gravity \cite{Surya} and more recently in the novel field of Lorentzian length spaces \cite{Kunzinger}.

 On the other hand, we have witnessed in the past decade a renewed interest for synthetic geometric methods, arising from their extensive use in scenarios where tools stemming from differential geometry are not available.  Indeed, several classical results from Riemannian geometry have been extended to the more general scope of length spaces. Basically, a length space is a metric space where the distance between any two points can be approached by the length of curves joining them. Remarkably, for geodesic length spaces we are able to define a synthetic notion of curvature by comparing geodesic triangles with triangles in a suitable space form of constant curvature (see \cite{Bridson,Burago,Plaut,Shiohama} and references therein).   The first attempts to establish a comparison theory for Lorentzian manifolds dates back to the work of Harris in the proof of Toponogov's Splitting Theorem \cite{Harris}.  In their seminal work \cite{Kunzinger}, Kunzinger and S\"amann developed a synthetic notion of (timelike) curvature bounds in a Lorentzian non-smooth context and used it to explore the nature of singularities. More recently, in \cite{Felix} we find the first detailed study of Lorentzian comparison theory, as well as further developments and techniques in the context of Lorentzian pre-length spaces. Moreover, in \cite{FS} the main analytic tools related to comparison theorems (such as properties of the exponential map) are laid out, along with fundamental results pertaining hyperbolic angles, most notably, the triangle inequality. In the present paper we lay out the basic tools to deal with comparisons in Lorentzian length spaces in a synthetic manner analog to the well established theory of Alexandrov or CAT spaces, as well as some global results. In particular, we discuss three equivalent notions of curvature bounds and a first variation formula for Lorentzian length spaces bounded from below by $0$. We hope these results will prove helpful in the endeavor of applying synthetic geometric methods in Relativity. 

The paper is organized as follows: in section \ref{sec:pre} we set the main definitions pertaining Lorentzian length spaces and fix the notation we will be using throughout this work. In section \ref{sec:comp} we recall the definition of bounded timelike curvature. After revising in section \ref{sec:nonnorm} the notion of non-normalized angle due to Alexander and Bishop \cite{Bishop}, we establish the equivalence of timelike curvature bounds and Alexandrov's convexity property in section \ref{sec:Alex}. On section \ref{sec:Toponogov} we define the notion of angle in Lorentzian pre-length spaces and use it to discuss the relation of timelike curvature bounds and the local Lorentzian version of Toponogov's property. Finally, in section \ref{sec:firstvar} we prove a global first variation formula for non-negative curvature bounded Lorentzian length spaces.

%%%%%%%%%%%%%%%%%%%%%%
%%%%%%%%%%%%%%%%%%%%%%  Preliminares
%%%%%%%%%%%%%%%%%%%%%%

\section{Preliminaries}\label{sec:pre}

Throughout this section we will recall some basic notions in the context of Lorentzian length spaces as defined in \cite{Kunzinger}. The first ingredient consist in an axiomatic formulation of causality, close in spirit to the original definition of causal spaces first proposed in \cite{KP}.

\begin{definition}\label{defi:prel}
A \emph{Lorentzian pre-length space} is a quintuple $(X,d,\ll,\leq, \tau)$ where  
\begin{enumerate}
\item $(X,d)$ is a metric space,
\item $\leq$ is a pre-order,
\item  $\ll$ is a transitive relation contained in $\leq$.
\item $\tau:X\times X \to [0,\infty]$ is lower semi-continuous function satisfying
\begin{itemize}
    \item $\tau(x,z)\geq \tau(x,y) + \tau(y,z)$ for all $x\leq y \leq z$
    \item $\tau(x,y)>0$ if and only if $x\ll y$.
\end{itemize}
\end{enumerate}
\end{definition}

If a pair of points satisfy $x\ll y$, ($x\le y$) we say that $x$ and $y$ are \emph{chronologically} (\emph{causally}) \emph{related}, respectively. \emph{Chronological} (\emph{causal}) \emph{future} and \emph{past sets} $I^+(x)$, $I^-(x)$ ($(J^+(y)$, $J^-(y)$) are thus defined in the standard way. The function $\tau$ is called a \emph{time separation} function.

Lorentzian pre-length spaces have just enough structure in order to establish some of the most basic facts pertaining causality. Most notably, the so called \emph{push-up property}  (if $x\ll y\le z$ or $x\le y\ll z$ then $x\ll z$) and the openness of the chronological sets $I^\pm (x)$.

A curve $\gamma:[a,b]\to X$ that is not constant on any subinterval of $[a,b]$ is called \emph{future-directed timelike} (\emph{causal}) if $\gamma$ is locally Lipschitz continuous with respect to $d$ and whenever $s,t\in [a,b]$ with $s<t$ then $\gamma(s)\ll \gamma(t)$ ($\gamma(s)\leq \gamma(t)$). Such curve is \emph{future-directed null} if it is causal and no pair of points on the curve are timelike related. Past directed curves are defined similarly. 

In order to have a sensible notion of length of causal curves, we rely on the time separation function. Thus, we define the $\tau-$\emph{length} of a future-directed causal curve $\gamma$ as 
\[
L_{\tau}(\gamma)= \inf\left\{ \displaystyle \sum_{i=0}^{n-1} \tau(\gamma(t_i),\gamma(t_{i+1})) : a=t_0<t_1<\cdots <t_n = b, n\in\mathbb{N} \right\}.
\]
where the infimum is taken over all possible partitions of $\gamma$. A future-directed causal curve $\gamma$ is \emph{maximal} if $L_{\tau}(\gamma)=\tau(\gamma(a),\gamma(b))$. Maximal curves are cornestone to synthetic geometry, as they are the closest analogs to geodesics.

Essentially, a \emph{Lorentzian length space} $(X,d,\ll,\leq, \tau)$ is a Lorentzian pre-length space with a local structure that resemble the one provided by normal neighborhoods and whose time separation function can be recovered from the $\tau$-length of curves. The former is achieved through the notions of \emph{localizing neighborhoods}: that is, neighborhoods $\Omega_x$ around each point $x\in X$ furnished with relations $\le_{\Omega_x}$, $\ll_{\Omega_x}$ and continuous functions $\omega_x:\Omega_x\times\Omega_x\to [0,\infty)$ such that
\begin{enumerate}
\item $(\Omega_x,d\vert_{\Omega_x\times\Omega_x},\ll_{\Omega_x},\leq_{\Omega_x}, \omega_x)$ is a Lorentzian pre-length space.
\item $I^\pm (y)\cap \Omega_x\neq\emptyset$, for all $y\in\Omega_x$.
\item All causal curves contained in $\Omega_x$ have uniformly bounded $d$-length.
\item For all $p\neq q\in \Omega_x$ with $p\le q$ there exists a future causal curve $\gamma_{pq}$ contained in $\Omega_x$ such that $L_\tau(\gamma_{pq})=\omega_x(p,q)$ and whose $\tau$-length is maximal among all future causal curves from $p$ to $q$ lying in $\Omega_x$.
\end{enumerate}

The precise definition reads as follows (refer to  Definition 3.22 of  \cite{Kunzinger}).

\begin{definition}\label{defi:lls}
A Lorentzian pre-length space for which:
\begin{enumerate}
    \item Every point $x$ has a localizing neighborhood $\Omega_x$.
    \item Every point has a neighborhood in which the $\ll$ is closed.
    \item If $x\le y$ then there exists a future causal curve from $x$ to $y$.
    \item $\tau (x,y)=\mathcal{T}(x,y)$, for all $x,y\in X$, where
\[
\mathcal{T}(x,y) = \sup\{L_{\tau}(\gamma): \gamma\mbox{ is future-directed causal curve from $x$ to $y$}\}.\footnote{Here, we set $\mathcal{T}(x,y)=0$ when the set of future-directed causal curves from $x$ to $y$ is empty.}
\]
is called a \emph{Lorentzian length space}.  
\end{enumerate}
\end{definition}

A causality theory for Lorentzian length spaces can be developed in a way that resembles the classical theory for spacetimes. In particular, a causal hierarchy can be established with the notion of global hyperbolicity at the top. Just as in the classical smooth case, a causal Lorentzian length space is \emph{globally hyperbolic} if the causal diamonds $J^+(x)\cap J^-(z)$ are compact. In this case, the time separation function $\tau : X\times X\to [0,\infty]$ is continuous and finite. Moreover,   $(X,d,\ll,\leq, \tau)$ satisfies the Avez-Seifert property: for any pair of causally related points $x\le y$ there exists a maximal future causal curve $\gamma$ from $x$ to $y$ \cite{ACS,Kunzinger}.

%%%%%%%%%%%%%%%%%%%%%%%%%%%%%%%%%%%%%%%%%%%%%%
%%%%%%%%%%%%%%%%%%%%%%%%%%%%%%%%%%%%%%%%%%%%%%  Triangle Comparison
%%%%%%%%%%%%%%%%%%%%%%%%%%%%%%%%%%%%%%%%%%%%%%

\section{Triangle Comparison}\label{sec:comp}

At the core of synthetic geometry is the notion of triangle comparison. In a nutshell, the main idea is that curvature bounds can be recovered locally from comparisons of the most basic geometric objects (lengths and angles) with respect to those found in a two dimensional model space.  As in the Alexandrov case, geodesic triangle comparison is the main key to describe the curvature in the Lorentzian context.  This is achieved by looking at timelike triangles in the Lorentzian space forms of constant sectional curvature $k$. We denote these models by
\[
\mathbb{M}_{k}^{L} = \left\{
\begin{array}{ll}
\mathbb{S}_{1}^{2}(r) & k=\frac{1}{r^2} \\
\mathbb{R}_{1}^{2} & k=0 \\
\mathbb{H}_{1}^{2}(r) & k=-\frac{1}{r^{2}}
\end{array}
\right.,
\]
where $\mathbb{S}_{1}^{2}(r)$ is the simply connected cover of two-dimensional de Sitter space, $\mathbb{R}_{1}^{2}$ is the two-dimensional Minkowski space and $\mathbb{H}_{1}^{2}(r)$ is the simple connected cover of the two-dimensional anti de Sitter space. 

Notice that in all these cases there exist restrictions on the lengths of the sides of a triangle akin to the triangle inequality in Euclidean geometry. In fact, since we will be dealing with triangles whose sides are unique maximizing timelike segments, the occurrence of conjugate points along geodesics hinders the possibility of having such kind of triangles with arbitrary long side lengths. Those restrictions are described in the Realizability Lemma (Lemma 4.6 of \cite{Kunzinger} or Lemma 2.1 in \cite{Bishop}), and if the side lengths of a triangle satisfy them, we would say that such triangle obey \emph{timelike size bounds for} $k$. Roughly speaking, for $k=0$ the restriction is given by the reverse triangle inequality, while for $k<0$  in addition we have that the greatest  side should be less than $\pi /\sqrt{-k}$. A timelike geodesic triangle in a model space ${M}_{k}^{L}$ whose vertices satisfy $x\ll y\ll z$ will be denoted by $\triangle xyz$.

\begin{definition}
A \emph{timelike geodesic triangle} $(x,y,z)$ in a Lorentzian length space $(X,d,\ll,\leq, \tau)$ is a triple of points in $X$ satisfying $x\ll y\ll z$ such that $\tau(x,z)<\infty$. Its \emph{sides} are  maximal future-directed causal curves $\alpha$ from $x$ to $y$, $\beta$ from $y$ to $z$ and $\gamma$ from $x$ to $z$, hence 
\[
L_{\tau}(\alpha)=\tau(x,y), \quad L_{\tau}(\beta)=\tau(y,z), \quad L_{\tau}(\gamma)=\tau(x,z).
\]
 A \emph{comparison triangle} for the geodesic triangle $(x,y,z)$ is a triangle $\triangle\bar{x}\bar{y}\bar{z}$ in a model space $\mathbb{M}_{k}^{L}$ with geodesic segments joining them $\bar{\alpha }$, $\bar{\beta}$, $\bar{\gamma}$ whose lengths equal those of $\alpha$, $\beta$, $\gamma$, respectively\footnote{Notice that comparison triangles are unique up to an isometry.}. In other words
 \[ 
\bar{\tau} (\bar x,\bar y) =\tau (x,y),\quad  \bar{\tau} (\bar y,\bar z) =\tau (y,z), \quad \bar{\tau} (\bar x,\bar z) =\tau (x,z).
 \]
 where $\bar{\tau}$ is the time separation function in the model space $\mathbb{M}_{k}^{L}$. If a point $q$ lies on a side of a timelike geodesic triangle, we denote by $\bar{q}$ the point lying on the corresponding side of the comparison triangle such that $\bar{\tau}(\bar{p},\bar{q})=\tau (p,q)$, where $p\in\{x,y\}$ is the initial point of the side.  
\end{definition}

%\begin{figure}[h]
%\centering{
%\def\svgwidth{250pt}
%\input{comparacion.pdf_tex}
%\caption{A timelike geodesic triangle and its comparison triangle in $\mathbb{M}_{k}^{L}$}
%}
%\end{figure}

\begin{figure}[h]
\centering{
\includegraphics[scale=.3]{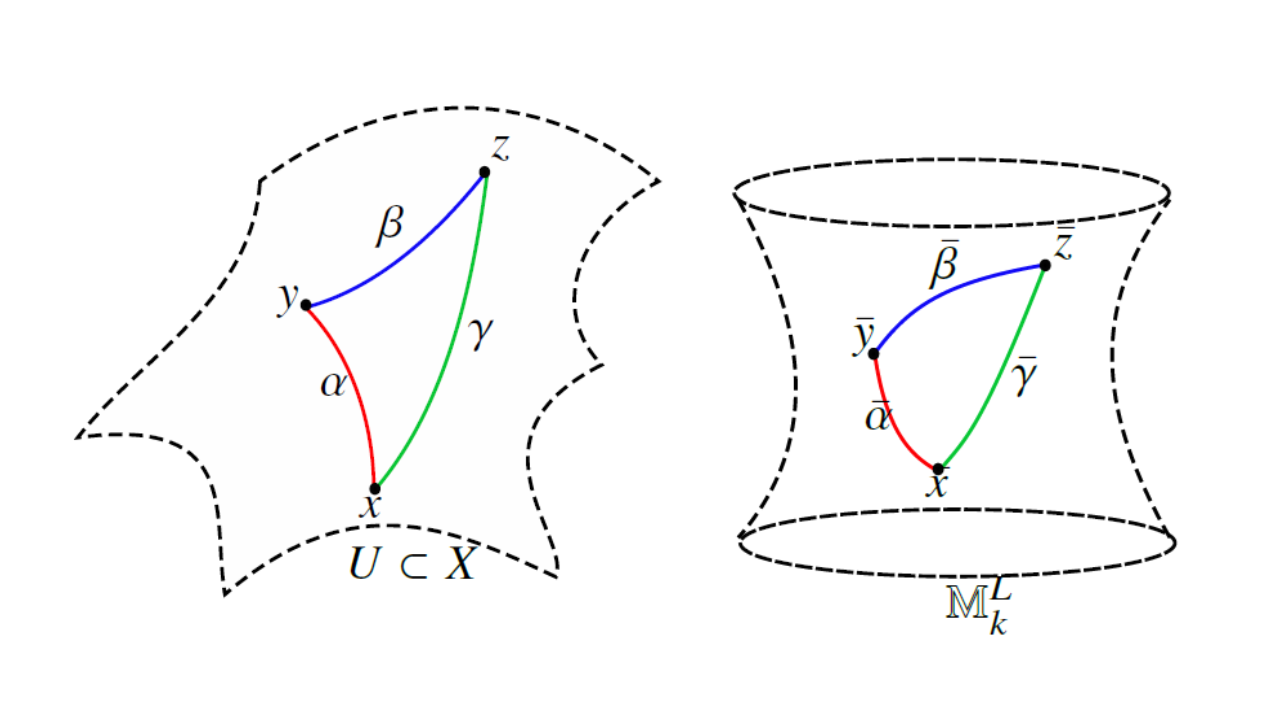}
\caption{A timelike geodesic triangle and its comparison triangle in $\mathbb{M}_{k}^{L}$}
}
\end{figure}

 We now state the original definition of timelike curvature bound for Lorentzian length spaces as established in \cite{Kunzinger}.

\begin{definition}\label{defi:curvbounds}
A Lorentzian pre-length space $(X,d,\ll,\leq,\tau)$ is said to have \emph{timelike curvature bounded below by} $k\in\mathbb{R}$ if around any $p\in X$ there exists a neighborhood $U\ni p$ with the following properties:
\begin{enumerate}
\item[(i)] $\tau\vert_{U\times U}$ is finite and continuous.
\item[(ii)] For every $x\ll y$ there exists a causal curve $\alpha$ in $U$ with $L_{\tau}(\alpha)=\tau(x,y)$.
\item[(iii)] For any timelike geodesic triangle $(x,y,z)$ in $U$, realized by maximal curves $\alpha$, $\beta$, $\gamma$ whose side lengths satisfy timelike size bounds for $k$ the following holds: if triangle $\triangle\bar{x}\bar{y}\bar{z})$ is a comparison triangle for $(x,y,z)$ in $\mathbb{M}_{k}^{L}$ realized by timelike geodesics $\bar{\alpha}$, $\bar{\beta}$ and $\bar{\gamma}$, then whenever $p$, $q$ are points on the sides of $(x,y,z)$ and $\bar{p},\bar{q}$ are the corresponding points in $(\bar{x},\bar{y},\bar{z})$, we have 
\[
\tau(p,q)\leq \bar{\tau}(\bar{p},\bar{q}).
\]
If under the same hypothesis the alternative inequality
\[
\tau(p,q)\ge \bar{\tau}(\bar{p},\bar{q}).
\]
holds, we will say that $(X,d,\ll,\leq,\tau)$ has \emph{timelike curvature bounded above by} $k$. The neighborhood $U$ is called a \emph{comparison neighborhood for} $p$. 
\end{enumerate}
\end{definition}

%\begin{figure}[h]
%\centering{
%\def\svgwidth{250pt}
%\input{CurvatureDefinition.pdf_tex}
%\caption{Timelike curvature bounded from below by $k$}
%}
%\end{figure}

\begin{figure}[h]
\centering{
\includegraphics[scale=.3]{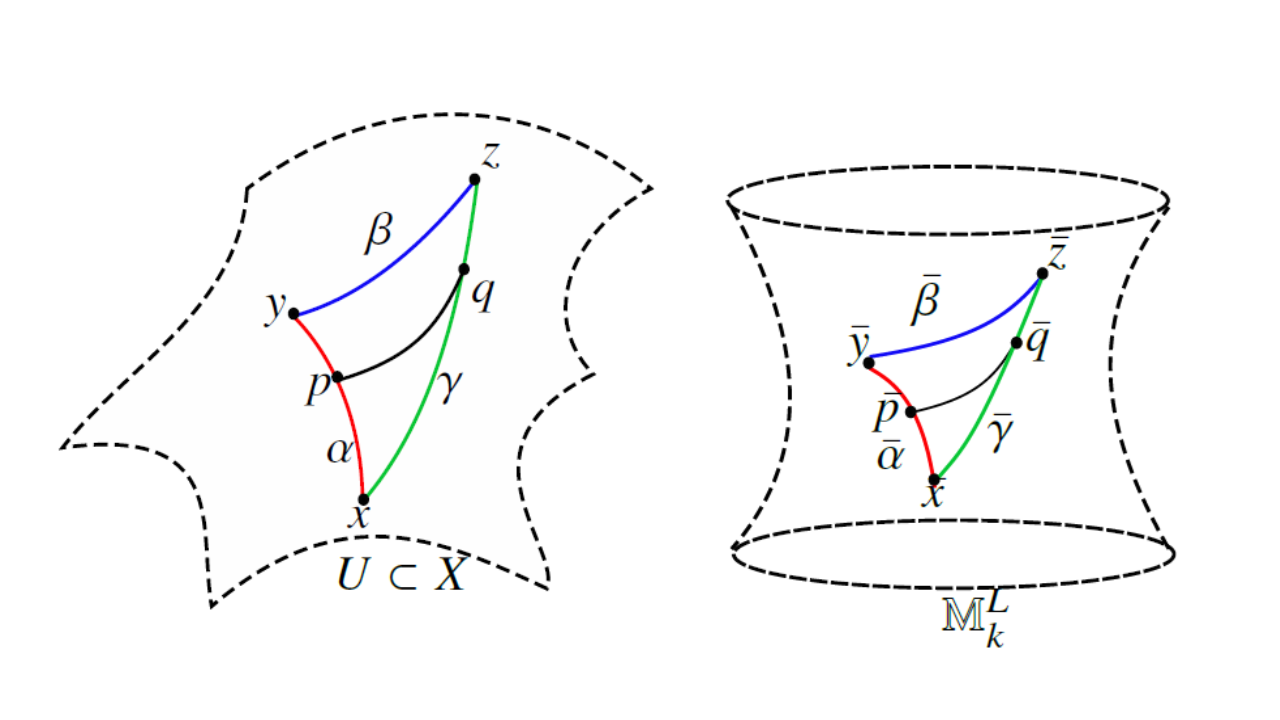}
\caption{Timelike curvature bounded from below by $k$}
}
\end{figure}

In the interest of having a lighter presentation, we include the technical aspects of Definition \ref{defi:curvbounds} as follows:

\begin{definition}\label{defi:curvneigh}
Let $(X,d,\ll,\leq,\tau)$ be a Lorentzian pre-length space. We say that $x\in X$ has a \emph{compatible neighborhood} $U$ if
\begin{enumerate}
    \item[(i)] $\tau\vert_{U\times U}$ is finite and continuous.
\item[(ii)] For every $x\ll y$ there exists a causal curve $\alpha$ in $U$ with $L_{\tau}(\alpha)=\tau(x,y)$.
%\item[(iii)] The lengths of the sides $\alpha$, $\beta$, $\gamma$ of any timelike geodesic triangle $(x,y,z)$ contained in $U$ satisfying timelike size bounds for $k$. 
\end{enumerate}
\end{definition}

\begin{remark}\label{GeodesicPropertiesU}
Let $U$ be a compatible neighborhood in a Lorentzian pre-length space  $(X,d,\ll,\leq,\tau)$.   
\begin{enumerate}
\item In virtue of  Proposition 3.34, Remark 4.3 and Remark 4.8 in \cite{Kunzinger}, any maximal timelike curve in $U$ can parameterized by arc length. Furthermore, any intermediate value of $\tau$ along $\alpha$, $\beta$ or $\gamma$ is attained. In particular, all of these follow when $(X,d,\ll,\leq,\tau)$ is a globally hyperbolic Lorentzian length space.
\item Let $(x,y,z)$ be a timelike geodesic triangle in $U$, realized by maximal causal curves $\alpha$, $\beta$, $\gamma$ whose side lengths satisfy timelike size bounds for $k$. If we take points $p\in \alpha$, $q\in \beta$,  $r\in\gamma$ then the sides of triangle $(p,y,q)$ also satisfy timelike size bounds for $k$. Moreover, if $p\ll r$, then the sides of triangle $(x,p,r)$ satisfy timelike size bounds for $k$ and so do the sides of triangle $(r,q,z)$ provided $r\ll q$. 
\end{enumerate}
\end{remark}

As can be readily seen, Definition \ref{defi:curvbounds} formalizes the intuitive notion that in the presence of positive (negative) timelike curvature, triangles look fatter (thinner) than flat triangles. In fact, this formulation agrees with the well known notions in metric geometry of Alexandrov (curvature bounded from below) and CAT (curvature bounded from above) spaces. Moreover, it is consistent with the definition of curvature bounds for semi-Riemannian manifolds proposed in \cite{Bishop}. However, there is a catch. According to  \cite{Bishop}, a Lorentzian manifold having curvature bounded from \emph{below} by $k$ while having timelike curvature bounded from \emph{below} (as a Lorentzian length space), it has sectional curvature on timelike planes bounded from \emph{above} by $k$. A similar statement holds if the words below and above are interchanged.

%%%%%%%%%%%%%%%%%%%%%%%%%%%%%%%%%
%%%%%%%%%%%%%%%%%%%%%%%%%%%%%%%%%    Example
%%%%%%%%%%%%%%%%%%%%%%%%%%%%%%%%%

Among the first applications of the synthetic notion of curvature in Lorentzian length spaces described in \cite{Kunzinger} we find the description of curvature singularities, a topic of great interest in the realm of Relativity. A Lorentzian length space has \emph{timelike curvature unbounded from below} (\emph{above}) if there exist a compatible neighborhood that fails to be a comparison neighborhood for all $k\in\mathbb{R}$. While this definition can be used to spot singularities (for instance, the interior region of Schwarszchild spacetime is a Lorentzian length space with timelike curvature unbounded from below), its use often requires comparisons on the large, which are at odds with the intuitive local character of curvature.  As an example, a timelike funnel with timelike $\lambda$ has timelike curvature unbounded from below (see Examples 3.19 and 4.21 in \cite{Kunzinger}), but both $I^-(p)$ and $I^+(q)$ are flat open subsets of it, hence Lorentzian length spaces in their own right with timelike curvature bounded ---both from below and from above-- by $0$.

Here we present a novel example of a globally hyperbolic Lorentzian length space with arbitrary small $k$-compatible neighborhoods in which no comparison is possible. Hence, in spite of being at the top of the causal ladder, any open subset of it has unbounded timelike curvature both from above and from below.

\begin{example}
Let us consider $(X,d)=(\mathbb{R}^2,d_T)$ where $d_T$ is the taxicab metric
\[
d_T((x_1,y_1),(x_2,y_2))=\vert x_1-x_2\vert +\vert y_1-y_2\vert ,
\] 
and let  the relations $\ll_T$, $\leq_T$ be the usual chronological and causal relations in Minkowski space $\mathbb{R}^2_1$. Furthermore,  for $(x_1,y_1),(x_2,y_2)\in\mathbb{R}^{2}$ we define
\[
\tau_T((x_1,y_1),(x_2,y_2) = \left\{
\begin{array}{ll}
y_2-y_1-\vert x_2-x_1\vert   & \mbox{if $(x_1,y_1)\leq (x_2,y_2)$} \\
0 & \mbox{otherwise}
\end{array}
\right.
\]
\end{example}
We first show now that $\mathbb{R}^{2,1}_T=(\mathbb{R}^{2},d_T,\ll_T,\leq_T,\tau_T)$ ---dubbed \emph{Lorentzian taxicab space}--- is a globally hyperbolic Lorentzian length space.

A straightforward computation shows that $\tau_T$ satisfies the causal properties of a time separation as described in Definition \ref{defi:prel}. Since $d_T$ is equivalent to the standard  Euclidean metric, the topology induced by $d_T$ is Euclidean, and as a consequence $d_T$ is continuous. Moreover, the class of (Lipschitz) causal curves in both $\mathbb{R}^{2,1}_T$ and $\mathbb{R}^2_1$ also coincide, which in turns implies that the causal diamonds of $\mathbb{R}^{2,1}_T$ are just the standard causal diamonds of $\mathbb{R}^{2}_1$. Compacity of the causal diamonds follows, and thus
$\mathbb{R}^{2,1}_T$ is a globally hyperbolic pre-length space.

Now we focus on the requirements of Definition \ref{defi:lls}. Global hyperbolicity implies causal connectivity and causal closedness. Moreover, given any point  $(x,y)\in \mathbb{R}^{2,1}_T$ consider the $d_T$ open ball centered at $(x,y)$
\[
\Omega_{(x,y)} = \{(p,q)\in \mathbb{R}^{2}: \vert p-x\vert +\vert q-y\vert <1\},
\]
and $\omega_{(x,y)}=\tau_T\vert_{\Omega_{(x,y)}}$. Conditions (1) and (2) of the definition of a localizing neighborhood are immediate. In order to show (3),  let $\gamma:[a,b]\to \mathbb{R}^{1,1}_T$, $\gamma(t)=(\gamma_1(t),\gamma_2(t))$,  be a future causal curve in $\Omega_{(x,y)}$ and take a partition $a=t_1<t_2<\cdots < t_N=b$. Since $\gamma_2(t_{i+1})- \gamma_2(t_i)\geq \vert\gamma_1(t_i)-\gamma_1(t_{i+1})\vert$ we have
\[
\begin{array}{rcl}
\dst\sum_{i=1}^{N-1} d_T(\gamma(t_i),\gamma(t_{i+1})) &=& \dst\sum_{i=1}^{N-1} \vert\gamma_1(t_i)-\gamma_1(t_{i+1})\vert + \vert\gamma_2(t_i)-\gamma_2(t_{i+1})\vert \\
&\leq& \dst\sum_{i=1}^{N-1} 2(\gamma_2(t_{i+1})- \gamma_2(t_i))\leq 4.
\end{array}
\]
Thus, the $d_T$ arc-length of curves is bounded.  

Even though $\mathbb{R}^2_1$ and $\mathbb{R}^{2,1}_T$ share topology and casuality, their geodesic structures are rather different. As opposed to the Minkowski case, the Lorentzian taxicab admits infinitely many maximal curves joining any pair of causally related points. Indeed,  given $(x_1,y_2)\le_T (x_2,y_2)$ take any future causal curve and a partition as above. Further assume $\gamma_1$ is a monotone function. Thus
\[
\begin{array}{rcl}
\dst\sum_{i=1}^{N-1} \tau_T(\gamma(t_i),\gamma(t_{i+1})) &=& \dst\sum_{i=1}^{N-1} \gamma_2(t_{i+1})-\gamma_2(t_{i})-\vert\gamma_1(t_{i+1})-\gamma_1(t_{i})\vert \\
&=& \dst\sum_{i=1}^{N-1} \gamma_2(t_{i+1})-\gamma_2(t_{i}) - \dst\sum_{i=1}^{N-1} \gamma_1(t_{i+1})-\gamma_1(t_{i}) \\
&=& \gamma_2(b)-\gamma_2(a)-(\vert\gamma_1(b)-\gamma_1(a)\vert )\\,
&=& \tau_T((x_1,y_1),(x_2,y_2))
\end{array}
\]
Hence $L_{\tau_T}(\gamma)=\tau_T((x_1,y_1),(x_2,y_2))$, and
\[
\tau_T((x_1,y_1),(x_2,y_2)) = L_{\tau_{T}}(\gamma) \leq \mathcal{T}((x_1,y_1),(x_2,y_2)) \leq  \tau_T((x_1,y_1),(x_2,y_2)).
\]
which shows that $\gamma$ is maximizing. As an immediate consequence, condition (4) of localizing neighborhoods holds and also $\mathcal{T}=\tau$. Thus $\mathbb{R}^{2,1}_T$ is a globally hyperbolic Lorentzian length space.

%\begin{figure}[ht]
%\centering{
%\def\svgwidth{200pt}
%\input{TaxiDiamante.pdf_tex}
%\caption{Infinitely many maximal causal curves joining causally related points in $\mathbb{R}^{2,1}_T$.}
%}
%\end{figure}

\begin{figure}[ht]
\centering{
\includegraphics[scale=.3]{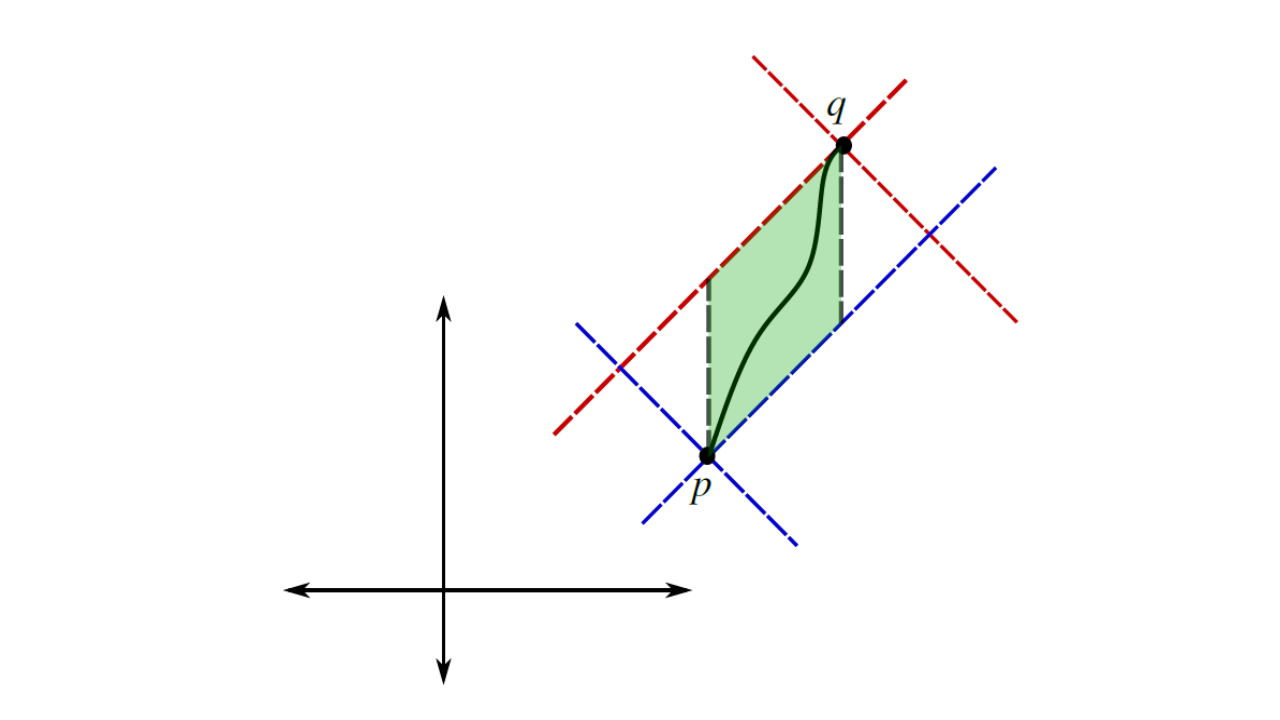}
\caption{Infinitely many maximal causal curves joining causally related points in $\mathbb{R}^{2,1}_T$.}
}
\end{figure}

We now apply directly the definition of timelike curvature bounds to show that there exists arbitrary small neighborhoods of $\mathbb{R}^{2,1}_T$ with no curvature bounds. We focus first at the case $k=0$. Let $\varepsilon>0$ and consider the triangle $(x,y, z)$ in $\mathbb{R}^{2,1}_T$ whose vertices are $x=(0,0)$, $y=(-2\varepsilon,3\varepsilon)$, $z=(\varepsilon,7\varepsilon)$,  and whose sides are the linear segments connecting them. Moreover, set the comparison triangle $\triangle\overline{x}\overline{y}\overline{z}$ in Minkowski space given by $\overline{x}=(0,0)$, $\overline{y}=(\sqrt{10}\varepsilon ,3\varepsilon)$ and $\overline{z}=(0,6\varepsilon)$. 

Further let  $q=( {\varepsilon}/{4},{13\varepsilon}/{2} )$ and notice it belongs to the segment joining $y$ with $z$. Its corresponding point in the comparison triangle $\triangle\overline{x}\overline{y}\overline{z})$ is $\overline{q}=( \sqrt{10}\varepsilon/4, {21\varepsilon}/{4} )$. Then
\[
\tau_T(x,q) =\dst\frac{25\varepsilon}{4} > \dst\frac{\sqrt{431}\varepsilon}{4} =
\overline{\tau}(\overline{x},\overline{q}) 
\]
On the other hand,  set $q=( -{5\varepsilon}/{4},4\varepsilon )$ on the segment from $y$ to $z$ and its corresponding point  $\overline{q}=( {3\sqrt{10}\varepsilon}/{4}, {15\varepsilon}/{4} )$. Thus
\[
\tau_T(x,q) = \frac{11\varepsilon}{4}< \frac{\sqrt{135}\varepsilon}{4} =
\overline{\tau}(\overline{x},\overline{q}) 
\]
Therefore, none of the curvature conditions of Definition \ref{defi:curvbounds} hold in a neighborhood $U$ containing a triangle isometric to $(x,y,z)$. 

We can use the same choice of triangle $(x,y,z)$ and points $p$, $q$ in order  to obtain similar inequalities in the model spaces $\mathbb{M}_k^L$ with $k\neq 0$ provided that $\varepsilon$ is small enough so that $(x,y,z)$ satisfies timelike size bounds.

\begin{remark}
The above example admits an straightforward generalization: consider a metric space $(X,d_X)$ and a Lorentzian pre-length space $(Y,d_Y,\ll_Y,\leq_Y,\tau_Y)$  let us take the metric $d:(X\times Y)\times (X\times Y)\to \mathbb{R}_{\geq 0}$ defined as
\[
d((a,b),(x,y)) = d_X(a,x) + d_Y(b,y).
\]
Let us set $\ll$, $\leq$ and $\tau:(X\times Y)\times (X\times Y)\to [0,\infty]$ defined as follows:
\begin{itemize}
\item $(a,b)\leq (x,y)$ if and only if $\tau_Y(b,y)\geq d_X(a,x)$ and $b\leq y$.
\item $(a,b)\ll (x,y)$ if and only if $\tau_Y(b,y)> d_X(a,x)$.
\item For $(a,b),(x,y)\in X\times Y$ we have
    \[
    \tau((a,b),(x,y)) = \left\{
    \begin{array}{ll}
    \tau_Y(b,y) -d_X(a,x)   & \mbox{if $(a,b)\leq (x,y)$} \\
    0 & \mbox{otherwise}
    \end{array}
    \right.
    \]
\end{itemize}
As can be checked, the \emph{taxicab Lorentzian product} $(X\times Y,d,\ll,\leq,\tau)$ is a Lorentzian pre-length space.
\end{remark}

%%%%%%%%%%%%%%%%%%%%%%%%%%%%%%%%%
%%%%%%%%%%%%%%%%%%%%%%%%%%%%%%%%%    Angles
%%%%%%%%%%%%%%%%%%%%%%%%%%%%%%%%%

\section{Non-normalized angles}\label{sec:nonnorm}

According to Euclidean geometry, side lengths and angle measure are the fundamental quantities associated to a triangle. Since the notion of timelike curvature bounds involves length comparison, it is natural to ask if there are alternative formulations involving angle measurements.  This is indeed the case for Alexandrov and CAT spaces. Moreover, in the context of semi-Riemannian geometry, an affirmative answer is given in Proposition 2.1 of \cite{Bishop}.  Refer to \cite{Kirch} for a thorough analysis on the properties of non-normalized angles.

\begin{definition}
Let $(M,\langle \cdot,\cdot\rangle )$  be a semi-Riemannian model space of curvature $k$. For a geodesic triangle $\triangle xyz$ in $M$ with $\alpha ,\gamma :[0,1]\to M$ geodesics  connecting $x$ with $y$, and $x$ with $z$, respectively, we denote $\measuredangle yxz= \langle \alpha'(0), \gamma'(0) \rangle$ and call it the \emph{non-normalized angle} at $p$.
\end{definition}

\begin{remark}
In the scenario depicted above, choose for instance a point $p=\alpha (\lambda)$, other than $x$ on the side $\alpha$. Then it follows 
\[
\angle pxz =\lambda \angle yxz.
\]
Hence,  $\angle pxz$ and $\angle yxz$ though not equal, only differ by the scaling factor $\lambda$. Hence the use of the term non-normalized angles is fully accurate.
\end{remark}

In view of the above remark, we can relate the non-normalized angles when the endpoints vary along the sides of a geodesic triangle. We state this relation in the form of a lemma, which will be used often in the following results.

\begin{lemma}\label{RescalAngulo}
Let $\triangle xyz$ be a timelike geodesic triangle in a Lorentzian model space $\mb{M}_{k}^{L}$ realized by maximal timelike curves $\alpha :[0,a]\to \mb{M}_{k}^{L} $, $\beta :[0,b]\to \mb{M}_{k}^{L} $, $\gamma :[0,c]\to \mb{M}_{k}^{L} $ whose side lengths satisfy timelike size bounds for $k$. Then, for every $a_0\in[0,a]$, $b_0\in[0,b]$ and $c_0\in[0,c]$ we have
\begin{enumerate}
\item[(a)] $\mangle\bar{\alpha}(a_0)\bar{x}\bar{z} = \frac{a_0}{a}\mangle \bar{y}\bar{x}\bar{z}$ and $\mangle \bar{y}\bar{x}\bar{\gamma}(c_0) = \frac{c_0}{c}\mangle \bar{y}\bar{x}\bar{z}$.
\item[(b)] $\mangle \bar{\beta}(b_0)\bar{y}\bar{x} = \frac{b_0}{b} \mangle \bar{z}\bar{y}\bar{x}$ and $\mangle \bar{z}\bar{y}\bar{\alpha}(a_0) = \frac{a-a_0}{a} \mangle \bar{z}\bar{y}\bar{x}$.
\item[(c)] $\mangle \bar{\gamma}(c_0)\bar{z}\bar{y} = \frac{c-c_0}{c} \mangle \bar{x}\bar{z}\bar{y}$ and $\mangle \bar{x}\bar{z}\bar{\beta}(b_0) = \frac{b-b_0}{b} \mangle\bar{x}\bar{z}\bar{y}$.
\end{enumerate}
\end{lemma}

Two of the main results pertaining the above notion are the Hinge Lemma (Lemma 2.2 in \cite{Bishop}) and the Straightening Lemma (Lemma 2.4 in \cite{Bishop}). We present these lemmas in a context adapted to our ends. Notice that in its original formulation, these results are stated using signed distances.

First, notice that the result below agrees completely with its basic Euclidean counterpart.

\begin{lemma}[Hinge Lemma for included angles]\label{HingeLemmaMk}
Let $\triangle x_1y_1z_1$ be two timelike geodesic triangles in $\mathbb{M}^{L}_{k}$ satisfying timelike curvature bounds for $k$. 
\begin{itemize}
    \item Suppose $\bar{\tau}(x_1,y_1)=\bar{\tau}(x_2,y_2)$ and $\bar{\tau}(y_1,z_1)=\bar{\tau}(y_2,z_2)$. Then $\bar{\tau}(y_1,z_1)\leq \bar{\tau}(y_2,z_2)$ if and only if $\measuredangle y_1x_1z_1 \leq \measuredangle y_2x_2z_2$.
    \item Suppose $\bar{\tau}(x_1,y_1)=\bar{\tau}(x_2,y_2)$ and $\bar{\tau}(y_1,z_1)=\bar{\tau}(y_2,z_2)$. Then $\bar{\tau}(x_1,z_1)\leq \bar{\tau}(x_2,z_2)$ if and only if
$\measuredangle x_1y_1z_1 \leq \measuredangle x_2y_2z_2$.
    \item Suppose $\bar{\tau}(x_1,z_1)=\bar{\tau}(x_2,z_2)$ and $\bar{\tau}(y_1,z_1)=\bar{\tau}(y_2,z_2)$. Then $\bar{\tau}(x_1,y_1)\leq \bar{\tau}(x_2,y_2)$ if and only if
$\measuredangle y_1z_1x_1 \leq \measuredangle y_2z_2x_2$.
\end{itemize}
 \end{lemma}

\begin{lemma}[Hinge Lemma for shoulder angles]\label{Hinge2}
Let $\triangle x_1y_1z_1$ be two timelike geodesic triangles in $\mathbb{M}^{L}_{k}$ satisfying timelike curvature bounds for $k$. \begin{itemize}
\item Suppose $\overline{\tau}(x_1,y_1)=\overline{\tau}(x_2,y_2)$ and $\overline{\tau}(x_1,z_1)=\overline{\tau}(x_2,z_2)$. If $\overline{\tau}(y_1,z_1)\leq \overline{\tau}(y_2,z_2)$ then $\measuredangle x_2y_2z_2 \leq \measuredangle x_1y_1z_1$ or $\measuredangle x_2z_2y_2\leq \measuredangle x_1z_1y_1$.
\item Suppose $\overline{\tau}(x_1,y_1)=\overline{\tau}(x_2,y_2)$ and $\overline{\tau}(y_1,z_1)=\overline{\tau}(y_2,z_2)$. If $\overline{\tau}(x_1,z_1)\leq \overline{\tau}(x_2,z_2)$ then $\measuredangle y_2x_2z_2\leq \measuredangle y_1x_1z_1$ or $\measuredangle y_2z_2x_2 \leq \measuredangle y_1z_1x_1$.
\item Suppose $\overline{\tau}(x_1,z_1)=\overline{\tau}(x_2,z_2)$ and $\overline{\tau}(y_1,z_1)=\overline{\tau}(y_2,z_2)$. If $\overline{\tau}(x_1,y_1)\leq \overline{\tau}(x_2,y_2)$ then $\measuredangle x_2y_2z_2\leq \measuredangle x_1y_1z_1$ or $\measuredangle y_2x_2z_2 \leq y_1x_1z_1$.
\end{itemize}
\end{lemma}

As a first application of the notion of non-normalized angles we prove an improved version of timelike curvature bounds more suited for applications. Namely, triangle comparison can be more easily performed when one of the points $p$ or $q$ in Definition \ref{defi:curvbounds} agree with a vertex, while the other point is chosen on its opposite side.

\begin{proposition} \label{CeviansCriteria}
Let $(X,d,\ll,\leq,\tau)$ be a Lorentzian pre-length space and  suppose that for every point $X$ there exists a compatible neighborhood $U$ with the following property: 
for any timelike geodesic triangle $(x,y,z)$ in $U$, realized by maximal timelike curves $\alpha$, $\beta$, $\gamma$ whose lengths satisfy timelike size bounds for $k$, whenever $p$ is a point on one side of $(x,y,z)$ and $v\in\{x,y,z\}$ is the vertex opposite to it, then we have 
    \begin{enumerate}
    \item[(i)] $\tau(v,p)\leq \bar{\tau}(\bar{v},\bar{p})$, if $\tau(v,p)>0$.
    \item[(ii)] $\tau(p,v)\leq \bar{\tau}(\bar{p},\bar{v})$, if
        $\tau(p,v)>0$;
    \end{enumerate}
 where $\triangle\bar{x}\bar{y}\bar{z}$ is a comparison triangle of $(x,y,z)$ in $\mb{M}_{k}^{L}$ with corresponding sides $\bar{\alpha}$, $\bar{\beta}$, $\bar{\gamma}$. Then $(X,d,\ll,\leq,\tau)$ has timelike curvature bounded below by $k$.
\end{proposition}

\begin{proof}
 We will prove that $U$ is a comparison neighborhood with respect to $\mb{M}_{k}^{L}$.  Thus, take two points $p$, $q$ on the sides of triangle $(x,y,z)$, so $\bar{p}$, $\bar{q}$ are the corresponding points on triangle $\triangle\bar{x}\bar{y}\bar{z}$.
 
 We first focus on vertex $y$. Suppose $p\in \beta$ and $q\in\alpha$. Notice that $0<\tau(q,p)<\overline{\tau} (\overline{q},\overline{p})$ and observe that $x\ll y\ll p$, then the sides of triangle $(x,y,p)$ satisfy timelike size bounds for $k$. Thus, let $\triangle x_1y_1p_1$ be a comparison triangle in $\mb{M}_{k}^{L}$ for triangle $(x,y,p)$.
%    \begin{figure}[h!]
%    \centering{
%    \def\svgwidth{250pt}
%    \input{VertexProposition.pdf_tex}
%    \caption{Triangle $(x,y,z)$ and comparison triangles $\triangle \bar{x}\bar{y}\bar{z}$ and $\triangle x_1y_1p_1$.}
%    }
%    \end{figure}

\begin{figure}[h!]
    \centering{
    \includegraphics[scale=.3]{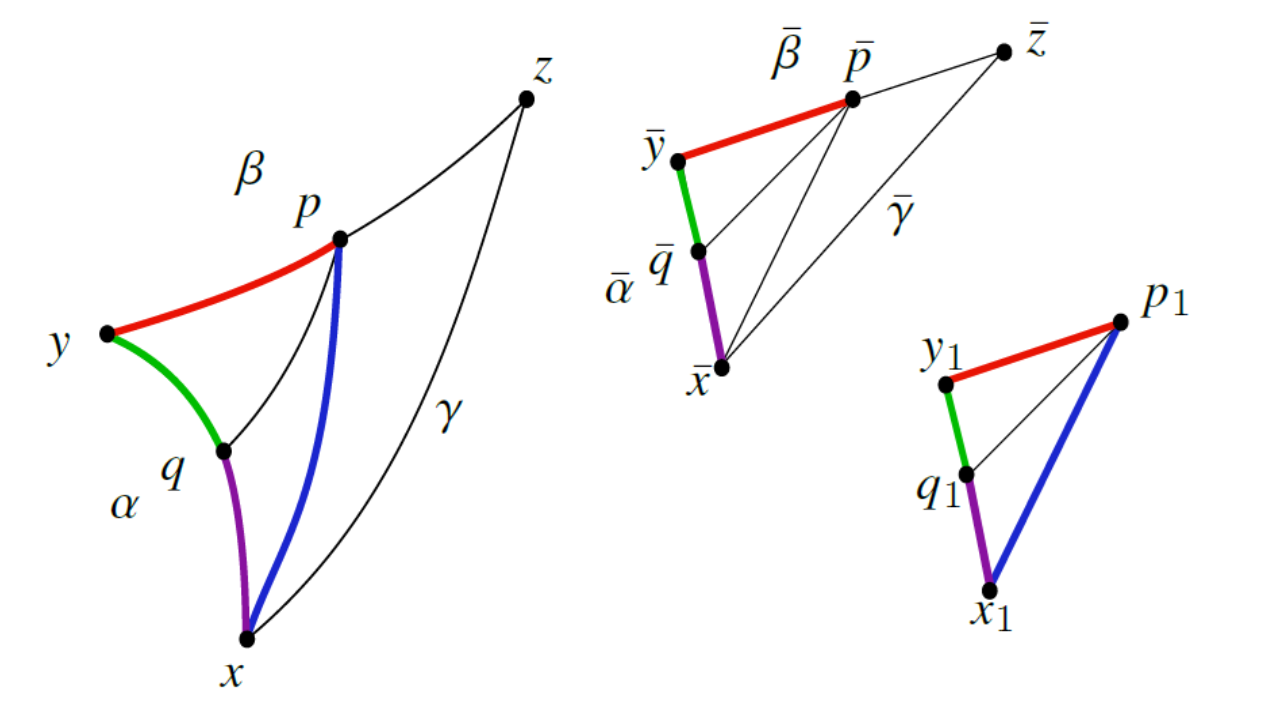}
    \caption{Triangle $(x,y,z)$ and comparison triangles $\triangle \bar{x}\bar{y}\bar{z}$ and $\triangle x_1y_1p_1$.}
    }
    \end{figure}

    By hypothesis we have $\tau(x,p)\leq \bar{\tau}(\bar{x},\bar{p})$. Then
    \[
    \bar{\tau}(x_1,p_1) = \tau(x,p) \leq \bar{\tau}(\bar{x},\bar{p}),
    \]
    and therefore by Lema \ref{HingeLemmaMk} we have $\measuredangle \bar{p}\bar{y}\bar{x}\geq \measuredangle p_1y_1x_1$. On the other hand, using Lemma \ref{RescalAngulo} we deduce that
    \[
    \measuredangle \bar{p}\bar{y}\bar{q} = \frac{\bar{\tau}(\bar{q},\bar{y})}{\bar{\tau}(\bar{x},\bar{y})} \measuredangle \bar{p}\bar{y}\bar{x} \ \mbox{and} \ \measuredangle p_1y_1q_1 = \dst\frac{\overline{\tau}(q_1,y_1)}{ \overline{\tau}(x_1,y_1)} \measuredangle p_1x_1y_1.
    \]
    Since $\bar{\tau}(\bar{x},\bar{y}) = \tau(x,y) = \bar{\tau}(x_1,y_1)$ and $\bar{\tau}(\bar{q},\bar{y}) = \tau(q,y) = \bar{\tau}(q_1,y_1)$ we conclude $\measuredangle \bar{p}\bar{y}\bar{q} \geq \measuredangle p_1y_1q_1$. Again, by Lemma \ref{HingeLemmaMk} we have $\bar{\tau}(q_1,p_1) \leq \bar{\tau}(\bar{q},\bar{p})$ and therefore
    \[
    \tau(q,p) \leq
    \bar{\tau}(q_1,p_1) \leq \bar{\tau}(\bar{q},\bar{p}).
    \]

Now we look at vertex $x$. Suppose $p\in\alpha$ and $q\in\gamma$. If $\tau(p,q)=0$ then $\tau(p,q)\leq \bar{\tau}(\bar{p},\bar{q})$. So take $\tau(p,q)>0$, which implies $x\ll p \ll q$ and the sides of triangle $(x,q,p)$ satisfy timelike size bounds for $k$. Let $(x_1,p_1,z_1)$ be a comparison triangle for $(x,p,z)$ and $q_1$ the corresponding point for $q$ in triangle $(x_1,p_1,z_1)$.
%    \begin{figure}[ht]
%    \centering{
%    \def\svgwidth{250pt}
%    \input{VertexPropositionPart2.pdf_tex}
%    \caption{Triangle $(x,y,z)$ and comparison triangles $%\triangle \bar{x}\bar{y}\bar{z}$ and $\triangle x_1p_1z_1$.}
%    }
%    \end{figure}
    
 \begin{figure}[ht]
    \centering{
    \includegraphics[scale=.3]{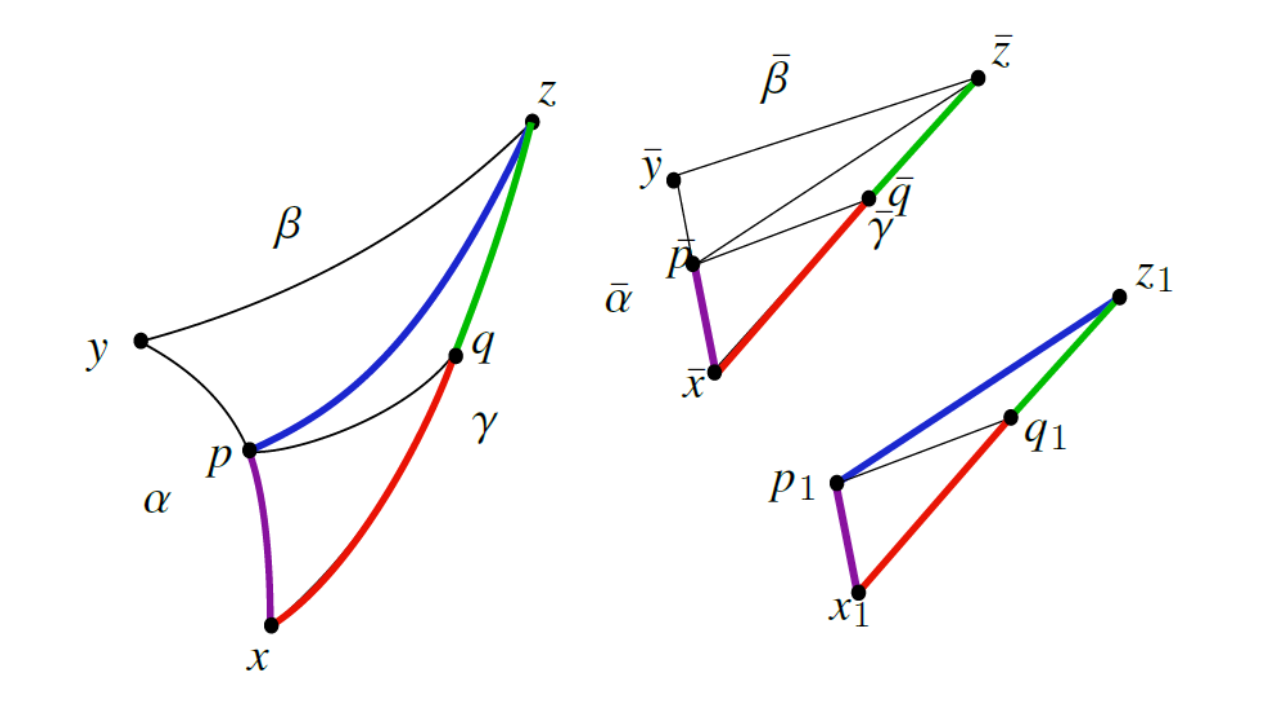}
    \caption{Triangle $(x,y,z)$ and comparison triangles $\triangle \bar{x}\bar{y}\bar{z}$ and $\triangle x_1p_1z_1$.}
    }
    \end{figure}    
    
    Then $\bar{\tau}(p_1,z_1) = \tau(p,z) \leq \bar{\tau}(\bar{p}, \bar{z})$, which implies $\mangle p_1x_1z_1 \leq \mangle \bar{p}\bar{x}\bar{z}$ because of Lemma \ref{HingeLemmaMk}. Following a  similar argument as in vertex $y$ we show that $\mangle p_1x_1q_1 \leq \mangle \bar{p}\bar{x}\bar{q}$ and therefore $\bar{\tau}(p_1,q_1)\leq \bar{\tau}(\overline{p}, \bar{q})$, again by Lemma \ref{HingeLemmaMk}. On the other hand $(x,p,z)$ we have $\tau(p,q) \leq \bar{\tau}(p_1,q_1)$, thus
    \[
    \tau(p,q) \leq \overline{\tau}(p_1,q_1) \leq \bar{\tau}(\bar{p}, \bar{q}).
    \]
    Again, the case $\tau(q,p)=0$ is trivial. If $\tau(q,p)>0$ we have $q\ll p\ll y$ and therefore this case follows an analog path as the case $\tau(p,q)>0$ just analyzed, where the point $y$ plays the role of $x$.  Thus $\tau(q,p)\leq \bar{\tau}(\bar{q},\bar{p})$.

Finally, notice that the analysis of vertex $z$ is completely analogous to the one preformed on vertex $x$, thus completing the proof. 
\end{proof}

%%%%%%%%%%%%%%%%%%%%%%%%%%%%
%%%%%%%%%%%%%%%%%%%%%%%%%%%%   Alexandrov Convexity
%%%%%%%%%%%%%%%%%%%%%%%%%%%%

\section{Alexandrov's convexity property}\label{sec:Alex}

 Let us recall that if two future directed timelike curves $\alpha,\beta :[0,a]\to M$ in a Lorentzian manifold $(M,\langle \cdot,\cdot\rangle )$ meet at a point $p=\alpha (0)=\beta (0)$, then the hyperbolic angle $\varphi$ spanned by $\alpha$ and $\beta$ at $p$ is given by the relation
\[
-\cosh \varphi = \frac{\langle \alpha^\prime (0),\beta^\prime (0)\rangle }{\vert \alpha^\prime (0)\vert \ \vert \beta^\prime (0)\vert}.
\]

In this section we provide analog formulations adapted to the context of Lorentzian pre-length spaces. The main idea is to construct a function that plays the same role as $-\cosh \varphi$, being its monotonicity the most relevant feature for comparison purposes.

\begin{definition}
Let $\triangle\bar{p}\bar{q}\bar{r}$ be a comparison triangle in $\mb{M}_{k}^{L}$ for a timelike geodesic triangle $(p,q,r)$ in a Lorentzian pre-length space $(X,d,\ll,\leq,\tau)$. We define the \emph{comparison angles at} $p$, $q$, $r$ by 
\[
\wangle_{k} rpq = \mangle \bar{r}\bar{p}\bar{q}\quad \wangle_{k} pqr = \mangle \bar{p}\bar{q}\bar{r}, \quad \wangle_{k} qrp = \mangle \bar{q}\bar{r}\bar{p},
\]
respectively.
\end{definition}

\begin{definition}\label{preAngle}
Given a compatible neighborhood $U$ in a Lorentzian pre-length space $(X,d,\ll,\leq,\tau)$, take a timelike geodesic triangle $(x,y,z)$ in $U$, realized by maximal causal curves  ${\alpha}:[0,a]\to X$, ${\beta}:[0,b]\to X$, ${\gamma}:[0,c]\to X$ satisfying timelike size bounds for $k$. We define the \emph{angle comparison functions} $\theta_{\alpha,\gamma}^{k}:(0,a]\times (0,c]\to \mb{R}$, $\theta_{\beta,\alpha}^{k}:(0,b]\times (0,a]\to \mb{R}$ and
$\theta_{\gamma,\beta}^{k}:(0,c]\times (0,b]\to \mb{R}$ as follows:
\begin{enumerate}
\item $\theta_{\alpha,\gamma}^{k}(s,t) = \dst\frac{\wangle_{k}\alpha(s)x\gamma(t)}{st}$, provided $\alpha(s)\ll \gamma(t)$ or $\gamma(t)\ll \alpha(s)$.
\item $\theta_{\gamma,\beta}^{k}(s,t) = \dst\frac{\wangle_{k}{\gamma}^*(s)z{\beta}^*(t)}{st}$, provided ${\gamma}^*(s)\ll {\beta}^*(t)$ or  ${\beta}^*(t)\ll {\gamma}^*(s)$.
\item $\theta_{\beta,\alpha}^{k}(s,t) = \dst\frac{\wangle_{k}\beta(s)y{\alpha}^*(t)}{st}$, for every $(s,t)\in (0,b]\times (0,a]$.
\end{enumerate}
where $\alpha^*:[0,a]\to X$, $\alpha (t)=\alpha (a-t)$ denotes the reverse curve of $\alpha$.
\end{definition}

It is immediate for the definition that $\theta_{\beta,\alpha}^{k}(s,t)>0$ and $\theta_{\alpha,\gamma}^{k}(s,t)<0$,  $\theta_{\gamma,\beta}^{k}(s,t)<0$.

\begin{remark}
Notice that because the relation $\ll$ is open and $\tau$ is continuous in $U$, we can always find small enough $s,t$ such that $\alpha(s)\ll \gamma(t)$ or $\gamma(t)\ll \alpha(s)$ so the conditions in part (1) of the above definition are met. The same applies for parts (2) and (3).
\end{remark}

An straightforward computation shows that when we apply Definition \ref{preAngle} to a geodesic triangle in $\mathbb{R}^2_1$, viewed as a Lorentzian pre-length space, we obtain that the value of $\theta_{\alpha,\gamma}^{0}(s,t)$ is constant ---independent of $(s,t)$--- and $\cosh^{-1}(-\theta_{\alpha,\gamma}^{0})$ is precisely the hyperbolic angle $\varphi$ described above. A similar scenario holds for the functions $\theta_{\alpha,\gamma}^{0}$ and $\theta_{\beta,\alpha}^{0}$. The same is true for the model spaces $\mb{M}_{k}^{L}$ with $k\neq 0$. Thus, these functions are natural candidates as comparison functions for Lorentzian length spaces.  To be able to show their monoticity, a lemma is in order.

\begin{lemma}\label{ComparisonPreLemma}
Let $(X,d,\ll,\leq,\tau)$ be a Lorentzian pre-length space and suppose it has timelike curvature bounded below by $k$. Let $U$ be a comparison neighborhood. Suppose that $(x,y,z)$ is a timelike geodesic triangle in $U$, realized by maximal causal curves $\alpha$, $\beta$, $\gamma$ whose lengths satisfy timelike size bounds for $k$.
\begin{enumerate}
\item For every $(s,t)\in (0,a]\times (0,c]$ such that $\alpha(s)\ll \gamma(t)$ or $\gamma(t)\ll \alpha(s)$ the following inequalities hold:
    \begin{enumerate}
    \item If $\alpha(s)\ll \gamma(t)$, then for all $s\geq s'$ and $t'\geq t$ we have
        \[
        \dst\frac{\wangle_{k} \alpha(s)x\gamma(t)}{s} \geq \dst\frac{\wangle_{k} \alpha(s')x\gamma(t)}{s'} \ \ \mbox{and} \ \
        \dst\frac{\wangle_{k} \alpha(s)x\gamma(t')}{t'} \geq \dst\frac{\wangle_{k} \alpha(s)x\gamma(t)}{t},
        \]
    \item If $\gamma(t)\ll \alpha(s)$, then for all $s'\geq s$ and $t\geq t'$ we have
        \[
        \dst\frac{\wangle_{k} \alpha(s')x\gamma(t)}{s'} \geq \dst\frac{\wangle_{k} \alpha(s)x\gamma(t)}{s} \ \ \mbox{and} \ \
        \dst\frac{\wangle_{k} \alpha(s)x\gamma(t)}{t} \geq \dst\frac{\wangle_{k} \alpha(s)x\gamma(t')}{t'},
        \]
    \end{enumerate}
\item For every $(s,t)\in (0,c]\times (0,b]$ such that ${\gamma}^*(s)\ll {\beta}^*(t)$ or ${\beta}^*(t)\ll {\gamma}^*(s)$ the following inequalities hold:
    \begin{enumerate}
    \item If ${\gamma}^*(s)\ll {\beta}^*(t)$, then for all $s\geq s'$ and $t'\geq t$ we have
        \[
        \dst\frac{\wangle_{k} {\gamma}^*(s)z{\beta}^*(t)}{s} \geq \dst\frac{\wangle_{k} {\gamma}^*(s')z{\beta}^*(t)}{s'} \ \ \mbox{and} \ \
        \dst\frac{\wangle_{k} {\gamma}^*(s)z{\beta}^*(t')}{t'} \geq \dst\frac{\wangle_{k} {\gamma}^*(s)z{\beta}^*(t)}{t},
        \]
    \item If  ${\beta}^*(t)\ll {\gamma}^*(s)$, then for all $s'\geq s$ and $t\geq t'$ we have
        \[
        \dst\frac{\wangle_{k} {\gamma}^*(s')z{\beta}^*(t)}{s'} \geq \dst\frac{\wangle_{k} {\gamma}^*(s)z{\beta}^*(t)}{s} \ \ \mbox{and} \ \
        \dst\frac{\wangle_{k} {\gamma}^*(s)z{\beta}^*(t)}{t} \geq \dst\frac{\wangle_{k} {\gamma}^*(s)z{\beta}^*(t')}{t'},
        \]
    \end{enumerate}
\item For every $(s,t)\in (0,b]\times t\in(0,a]$ we have
    \[
    \dst\frac{\wangle_{k} \beta(s)y{\alpha}^*(t)}{s} \geq \dst\frac{\wangle_{k} \beta(s')y{\alpha}^*(t)}{s'} \ \ \mbox{and} \ \
    \dst\frac{\wangle_{k} \beta(s)y{\alpha}^*(t')}{t'} \geq \dst\frac{\wangle_{k} \beta(s)y{\alpha}^*(t)}{t},
    \]
    for all $s\geq s'$ and $t'\geq t$.
\end{enumerate}
\end{lemma}

\begin{proof}
For (1a), take $s\geq s'$, then $x\ll \alpha(s')\ll \alpha(s) \ll \gamma(t)$ and therefore the sides of triangle $(x,\alpha(s'),\gamma(t))$ satisfy timelike size bounds for $k$. Let $\triangle x_1p_1q_1)$ and $\triangle x_2p_2q_2$ be  comparison triangles for $(x,\alpha(s),\gamma(t))$ and $(x,\alpha(s'),\gamma(t))$, respectively.
%\begin{figure}[ht]
%    \centering{
%    \def\svgwidth{250pt}
%    \input{AlexandrovConvexity1.pdf_tex}
%    \caption{Triangle $(x,y,z)$ and comparison triangles $%\triangle x_1p_1q_1$ and $\triangle x_2p_2q_2$.}
%    }
%\end{figure}

\begin{figure}[ht]
    \centering{
    \includegraphics[scale=.3]{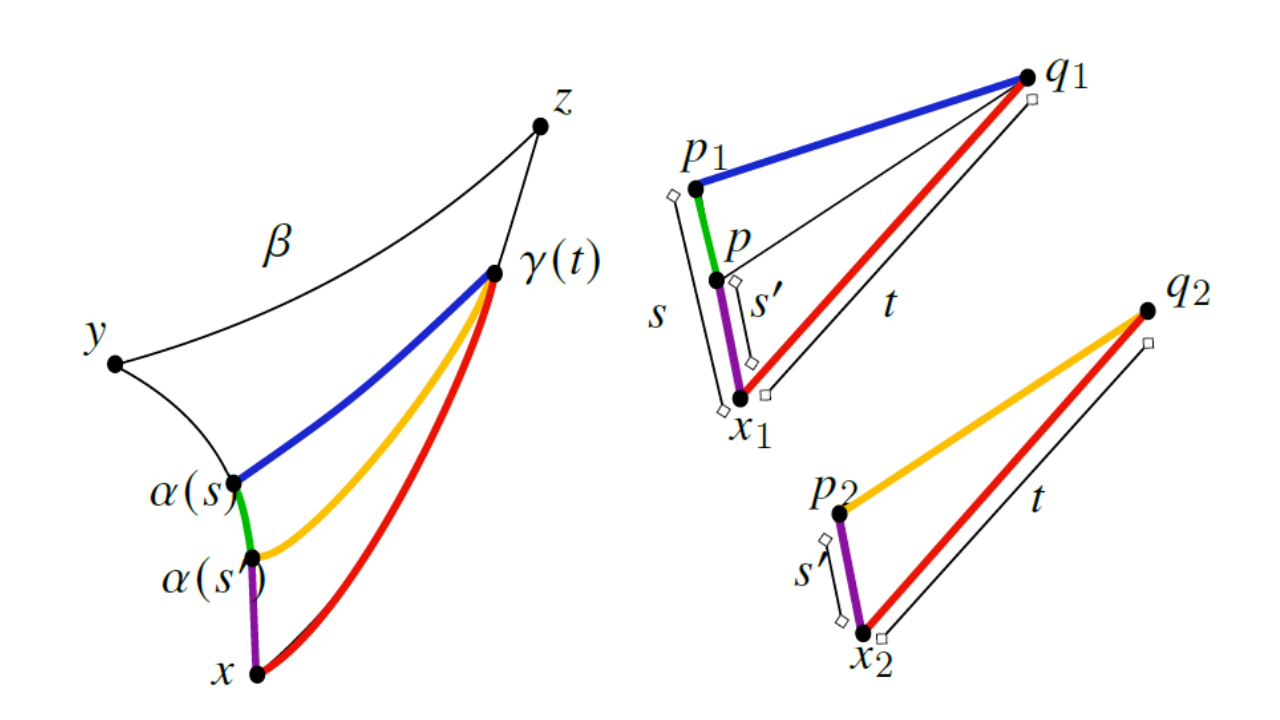}
    \caption{Triangle $(x,y,z)$ and comparison triangles $\triangle x_1p_1q_1$ and $\triangle x_2p_2q_2$.}
    }
\end{figure}

Let $p$ be the corresponding point in triangle $\triangle x_1p_1q_1$ for $\alpha(s')$ and denote $\theta_1= \measuredangle p_1x_1q_1$, $\theta_2  = \measuredangle p_2x_2q_2$. Since $U$ is a comparison neighborhood with respect to $\mb{M}_{k}^{L}$ we have
\[
\bar{\tau}(p,q_1) \geq \tau(\alpha(s'),\gamma(t)) = \bar{\tau}(p_2,q_2),
\]
thus by Lemma \ref{HingeLemmaMk} we get $\measuredangle px_1q_1 \geq \measuredangle p_2x_2q_2=\theta_2$. By Lemma \ref{RescalAngulo} we have $\dst\frac{s'}{s}\theta_1 = \mangle_k px_1q_1 \geq \theta_2$. This implies
\[
\dst\frac{\wangle_{k} \alpha(s)x\gamma(t)}{s} \geq \dst\frac{\wangle_{k} \alpha(s')x\gamma(t)}{s'}.
\]
For part of (1b), fix $t'\geq t$. Then $x\ll \alpha(s) \ll \gamma(t) \ll \gamma(t')$ and the sides of $(x,\alpha(s),\gamma(t'))$ satisfy timelike size bounds for $k$. Let $\triangle x_1p_1q_1$ and $\triangle x_2p_2q_2$ be comparison triangles for triangles $(x,\alpha(s),\gamma(t'))$ and $(x,\alpha(s),\gamma(t))$, respectively.
%\begin{figure}[ht]
%    \centering{
%    \def\svgwidth{250pt}
%    \input{AlexandrovConvexity2.pdf_tex}
%    \caption{Triangle $(x,y,z)$ and comparison triangles $%\triangle x_1p_1q_1$ and $\triangle x_2p_2q_2$.}
%    }
%\end{figure}
\begin{figure}[ht]
    \centering{
    \includegraphics[scale=.3]{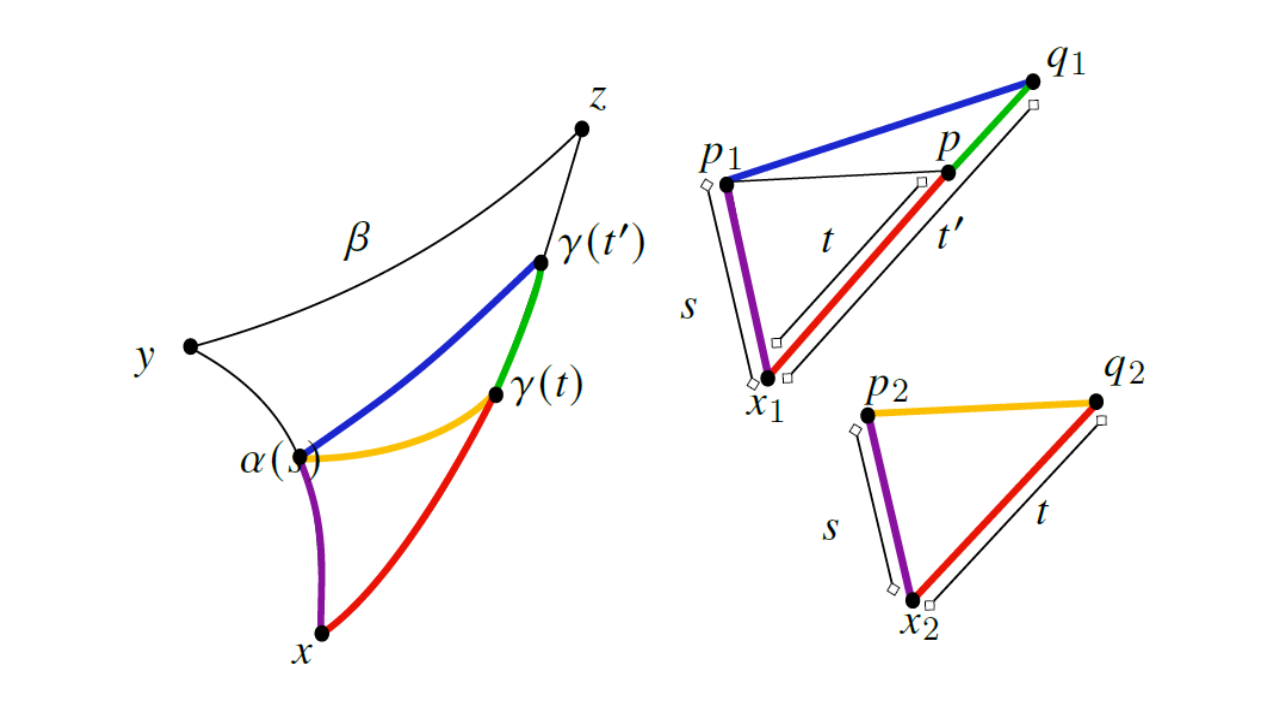}
    \caption{Triangle $(x,y,z)$ and comparison triangles $\triangle x_1p_1q_1$ and $\triangle x_2p_2q_2$.}
    }
\end{figure}
Denote by $p$ the corresponding point for $\gamma(t)$ in triangle $\triangle x_1p_1q_1$ and take $\theta_1 = \measuredangle p_1x_1q_1$, $\theta_2=\measuredangle p_2x_2q_2$. Just like in the previous case we have
\[
\bar{\tau}(p_1,p) \geq \tau(\alpha(s),\gamma(t)) = \bar{\tau}(p_2,q_2)
\]
and applying Lemma \ref{HingeLemmaMk} we deduce $\measuredangle p_1x_1p \geq \measuredangle p_2x_2q_2 = \theta_2$. Notice $\dst\frac{t}{t'}\theta_1 = \measuredangle p_1x_1p \geq \theta_2$ by Lemma \ref{RescalAngulo}, which implies
\[
\dst\frac{\wangle_{k} \alpha(s)x\gamma(t')}{t'} \geq  \dst\frac{\wangle_{k} \alpha(s)x\gamma(t)}{t}.
\]

Alternatively, if  $\gamma(t)\ll \alpha(s)$, then $\gamma(t)\ll \alpha(s')$ for all $s'\geq s$. Let us take the comparison triangles $\triangle x_1q_1p_1$ and $\triangle x_2q_2p_2$ for triangles $(x,\gamma(t),\alpha(s'))$ and $(x,\gamma(t),\alpha(s))$, respectively. Also, the point $p$ is the corresponding one for $\alpha(s)$ in triangle $\triangle x_1q_1p_1$. Thus by curvature conditions we have
\[
\overline{\tau}(q_2,p_2) = \tau(\gamma(t),\alpha(s)) \leq \overline{\tau}(\gamma(t),\alpha(s')).
\]
Therefore
\[
\frac{s}{s'} \wangle_{k} \alpha(s')x\gamma(t) = \frac{s}{s'}\measuredangle p_1x_1q_1 = \measuredangle px_1q_1 \geq \measuredangle p_2x_2q_2 = \wangle_{k} \alpha(s)x\gamma(t),
\]
and hence  
\[\dst\frac{\wangle_{k} \alpha(s')x\gamma(t)}{s'} \geq \dst\frac{\wangle_{k} \alpha(s)x\gamma(t)}{s}.\] 
On the other hand, if $t\geq t'$ then $\gamma(t')\ll \alpha(s)$. Choose comparison triangles $\triangle x_1q_1p_1$ and $\triangle x_2q_2p_2$ for triangles $(x,\gamma(t),\alpha(s))$ and $(x,\gamma(t'),\alpha(s))$, respectively. Now, the point $p$ on triangle $\triangle x_1q_1p_1$ the corresponding point for $\gamma(t')$. Since
\[
\overline{\tau}(q_2,p_2) = \tau(\gamma(t'),\alpha(s)) \leq \overline{\tau}(p,p_1),
\]
we obtain
\[
\frac{t'}{t}\wangle_{k} \alpha(s)x\gamma(t) =  \frac{t'}{t}\measuredangle p_1x_1q_1 = \measuredangle p_1x_1p \geq \measuredangle p_2x_2q_2 =  \wangle_{k} \alpha(s)x\gamma(t').
\]
In conclusion 
\[\dst\frac{\wangle_{k} \alpha(s)x\gamma(t)}{t} \geq \dst\frac{\wangle_{k} \alpha(s)x\gamma(t')}{t'}.\]

Now, let us focus in cases (3a) and (3b). If $s\geq s'$ then ${\alpha}^*(t)\ll \beta(s')$. Let us take the comparison triangles $\triangle p_1y_1q_1$ and $\triangle p_2y_2q_2$ for triangles $({\alpha}^*(t), y, \beta(2))$ and $({\alpha}*(t), y, \beta(s'))$, respectively (here observe that $\tau({\alpha}^*(t), y)=t$ by definition). Set $p$ the corresponding point for $\beta(s')$ in triangle $\triangle p_1y_1q_1$. Then
\[
\overline{\tau}(p_2,q_2) = \tau({\alpha}^*(t),\beta(s')) \leq \overline{\tau}(p_1,p),
\]
it means
\[
\frac{s'}{s}\wangle \beta(s)y{\alpha}^*(t) = \frac{s'}{s}\measuredangle p_1y_1q_1 = \measuredangle p_1y_1p \geq \measuredangle p_2y_2q_2 = \wangle \beta(s')y{\alpha}^*(t). 
\]
Hence 
\[\dst\frac{\wangle_{k} \beta(s)y{\alpha}^*(t)}{s} \geq \dst\frac{\wangle_{k} \beta(s')y{\alpha}^*(t)}{s'}.\]
In case that $t'\geq t$ we take the comparison triangles $\triangle p_1y_1q_1$ and $\triangle p_2y_2q_2$ for triangles $({\alpha}^*(t'), y, \beta(s))$ and $(\alpha(t), y,\beta(s))$, respectively (here $\tau({\alpha}^*(t'),y)=t'$ and $\tau({\alpha}^*(t),y)=t$). If $p$ is the corresponding point for $\alpha(t)$ in triangle $\triangle p_1y_1q_1$, then
\[
\overline{\tau}(p_2,q_2) = \tau({\alpha}^*(t),\beta(s)) \leq \overline{\tau}(p,q_1).
\]
This last inequality implies
\[
\frac{t}{t'}\wangle {\alpha}^*(t')y\beta(s) = \frac{t}{t'}\measuredangle p_1y_1q_1 =\measuredangle py_1q_1 \geq \measuredangle p_2y_2q_2 = \wangle {\alpha}^*(t)y\beta(s),
\]
thus 
\[\dst\frac{\wangle_{k} \beta(s)y\tilde{\alpha}(t')}{t'} \geq \dst\frac{\wangle_{k} \beta(s)y\tilde{\alpha}(t)}{t}.
\]

Finally, cases (2a) and  (2b) are analogous to (1a) and (1b) and the proof is complete.

\end{proof}

In Alexandrov geometry the monotonicity of the angle comparison functions is equivalent to the definition of curvature bounds (see for example Definition 4.3.1 of \cite{Burago} or Section 2.2 of \cite{Shiohama}). Hence, this monotonicity property is termed as \emph{the local version of the Alexandrov convexity}. We proceed to establish a similar result in the Lorentzian context.

\begin{theorem}[Angle monotonicity] \label{MonotonicityCriterion}
Let $(X,d,\ll,\leq,\tau)$ be a Lorentzian pre-length space and suppose it has timelike curvature bounded below by $k$. Let $U$ be a comparison neighborhood. Suppose that $(x,y,z)$ is a timelike geodesic triangle in $U$, realized by maximal causal curves $\alpha$, $\beta$, $\gamma$ whose lengths satisfy timelike size bounds for $k$. For every $s,s'\in (0,\tau(x,y)]$ and $t,t'\in(0,\tau(x,z)]$ such that $s'\leq s$, $t'\leq t$ and
\begin{enumerate}
\item  $\alpha(s)\ll \gamma(t)$ or $\gamma(t) \ll \alpha(s)$,
\item  $\alpha(s')\ll \gamma(t')$ or $\gamma(t') \ll \alpha(s')$,
\end{enumerate}
we have the following monotonicity condition
\[
\theta^k_{\alpha,\gamma}(s',t')\leq \theta^k_{\alpha,\gamma}(s,t).
\]
Similar inequalities apply for functions $\theta^{k}_{\gamma,\beta}$ and $\theta^{k}_{\beta,\alpha}$.
\end{theorem}

\begin{proof}
Here we will deal with the case when $\alpha(s)\ll \gamma(t)$ and $\alpha(s')\ll \gamma(t')$ and similar ideas apply for the other cases. Thus, using Proposition \ref{ComparisonPreLemma} we get
\[
t\theta^{k}_{\alpha,\gamma}(s,t) =\dst\frac{\wangle_{k} \alpha(s)x\gamma(t)}{s} \geq \dst\frac{\wangle_{k} \alpha(s')x\gamma(t)}{s'} = t\theta^{k}_{\alpha,\gamma}(s',t),
\]
which implies $\theta^{k}_{\alpha,\gamma}(s,t)\geq \theta^{k}_{\alpha,\gamma}(s',t)$. On the other hand
\[
s'\theta^{k}_{\alpha,\gamma}(s',t) = \dst\frac{\wangle_{k} \alpha(s')x\gamma(t)}{t} \geq \dst\frac{\wangle_{k} \alpha(s')x\gamma(t')}{t'}= s'\theta^{k}_{\alpha,\gamma}(s',t'),
\]
and therefore $\theta^{k}_{\alpha,\gamma}(s',t)\geq \theta^{k}_{\alpha,\gamma}(s',t')$. Finally
\[
\theta^{k}_{\alpha,\gamma}(s,t)\geq \theta^{k}_{\alpha,\gamma}(s',t) \geq \theta^{k}_{\alpha,\gamma}(s',t').
\]
\end{proof}

As in the Alexandrov case, by assuming the monotonicity of $\theta^{k}_{\alpha,\gamma}$, $\theta^{k}_{\gamma,\beta}$ $\theta^{k}_{\beta,\alpha}$ we get the converse of Theorem \ref{MonotonicityCriterion}.

\begin{theorem}
Let $(X,d,\ll,\leq,\tau)$ be a Lorentzian pre-length space 
and  $(x,y,z)$ be a timelike geodesic triangle in a compatible neighborhood $U$, realized by maximal timelike curves $\alpha$, $\beta$, $\gamma$ whose lengths satisfy timelike size bounds for $k$. If $\theta^{k}_{\alpha,\gamma}$, $\theta^{k}_{\gamma,\beta}$ and $\theta^{k}_{\beta,\alpha}$ are increasing functions, then $(X,d,\ll,\leq,\tau)$ has timelike curvature bounded below by $k$.
\end{theorem}

\begin{proof}
We rely on Proposition  \ref{CeviansCriteria} in order to prove that $(X,d,\ll,\leq,\tau)$ has timelike curvature bounded below by $k$.  Now suppose $p\in \alpha$, then $p\ll y \ll z$ and $\bar{p}\in\bar{\alpha}$ so that $\tau(x,p)=\bar{\tau}(\bar{x},\bar{p}) = s$. Let $\triangle x_1p_1z_1$ be  a comparison triangle for $(x,p,z)$ in $\mb{M}_{k}^{L}$.
%\begin{figure}[ht]
%\centering{
%\def\svgwidth{250pt}
%\input{MonoticityTheorem1.pdf_tex}
%\caption{Triangle $(x,y,z)$ and comparison triangles $\triangle %\bar{x}\bar{y}\bar{z}$ and $\triangle x_1p_1z_1$.}
%}
%\end{figure}
\begin{figure}[ht]
\centering{
\includegraphics[scale=.3]{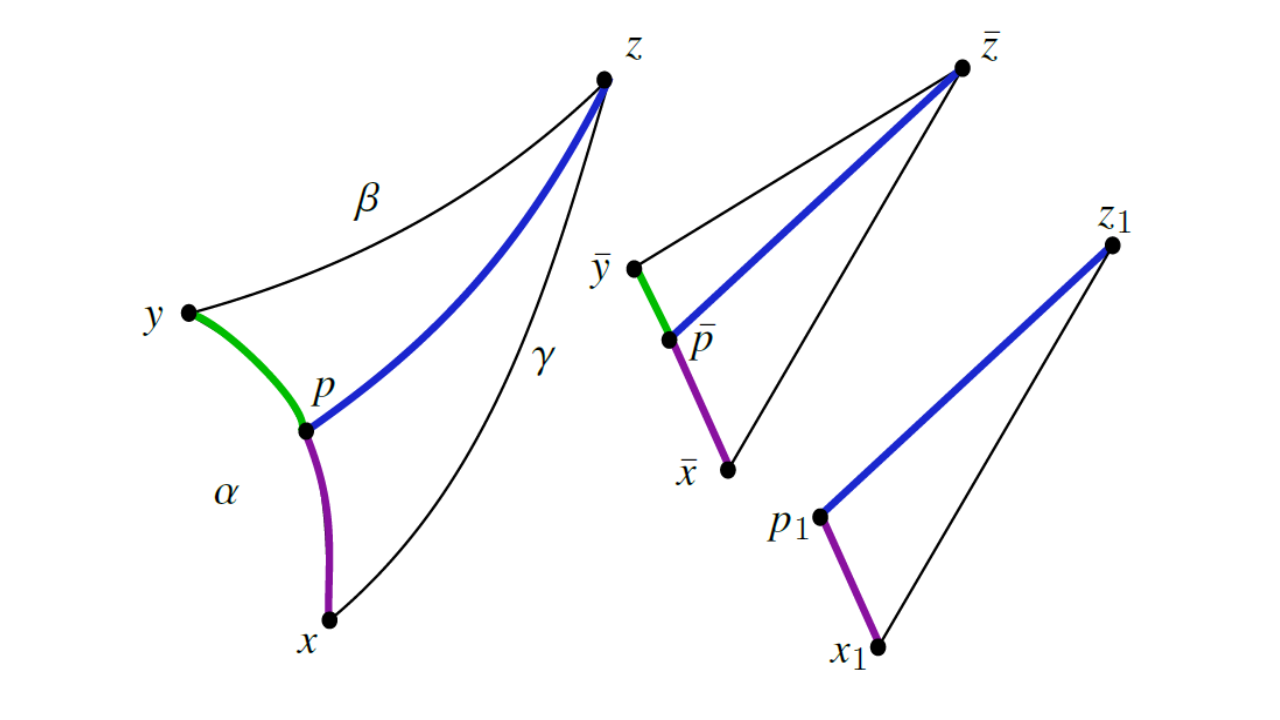}
\caption{Triangle $(x,y,z)$ and comparison triangles $\triangle \bar{x}\bar{y}\bar{z}$ and $\triangle x_1p_1z_1$.}
}
\end{figure}
Note that
\[
\dst\frac{\mangle p_1x_1z_1}{sc} =  \dst\frac{\wangle_{k} pxz }{sc} =\theta^{k}_{\alpha,\gamma}(s,c),
\]
and due to the monotonicity of $\theta^{k}_{\alpha,\gamma}$ we conclude
\[
\dst\frac{\mangle p_1x_1z_1}{sc} = \theta^{k}_{\alpha,\gamma}(s,c)
\leq \theta^{k}_{\alpha,\gamma}(a,c)
= \dst\frac{\wangle_{k} yxz}{ac}
= \dst\frac{\mangle \bar{y}\bar{x}\bar{z}}{ac},
\]
Therefore $\dst\frac{\mangle p_1x_1z_1}{sc} \leq \dst\frac{\mangle \bar{y}\bar{x}\bar{z}}{ac}$. On the other hand, by Lemma \ref{RescalAngulo} it follows that
\[
\mangle \bar{y}\bar{x}\bar{z} = \frac{a}{s}\mangle \bar{p}\bar{x}\bar{z},
    \]
thus
\[
\dst\frac{\mangle p_1x_1z_1}{sc} \leq \dst\frac{\mangle \bar{y}\bar{x}\bar{z}}{ac} = \dst\frac{\mangle \bar{p}\bar{x}\bar{z}}{sc},
\]
then $\mangle p_1x_1z_1\leq \mangle \bar{p}\bar{x}\bar{z}$. Finally, applying Lemma \ref{HingeLemmaMk} in $\mb{M}_{k}^{L}$ we get the desired inequality $\tau(p,z)\leq \bar{\tau}(\bar{p},\bar{z})$. All the remaining cases can be worked out similarly.
\end{proof}

%%%%%%%%%%%%%%%%%%%%%%%%%%%%%%%%%%%%
%%%%%%%%%%%%%%%%%%%%%%%%%%%%%%%%%%%%   Toponogov
%%%%%%%%%%%%%%%%%%%%%%%%%%%%%%%%%%%%

\section{Normalized angles in Lorentzian pre-length spaces}\label{sec:Toponogov}

So far, we have been able to find suitable angle comparison functions in our non-smooth context. Now we are ready to define the main notion of our work, namely, normalized angles for Lorentzian pre-length spaces.

\begin{definition}\label{LorentzianAngle}
Following the same assumptions of Definition \ref{preAngle} we set
\[
\measuredangle yxz = \dst\lim_{s,t\to 0} \theta^{k}_{\alpha,\gamma}(s,t), \ \measuredangle xyz = \dst\lim_{s,t\to 0} \theta^{k}_{\beta,\alpha}(s,t)
\]
\[
\mbox{ and } \ \measuredangle xzy = \dst\lim_{s,t\to 0} \theta^{k}_{\gamma,\beta}(s,t).
\]
\end{definition}

\begin{remark}\label{ExistenceAngle}
Notice that the limits for $\measuredangle yxz$ and $\measuredangle xzy$  described above need not exist, even in the presence of lower curvature bounds, since in this case the limits may diverge to $-\infty$. On the other hand, since $\theta^k_{\alpha ,\beta}(s,t)$ is positive, monotonicity guarantees that $\measuredangle xyz$ always exists.

Furthermore, as we  will prove, all normalized angles exist when the vertices of a geodesic triangle are not conjugate points along their sides. In other words if in a timelike geodesic triangle the sides can be extended past their vertices as maximizing timelike curves then the normalized angles exist (refer to Proposition \ref{SumGreaterCero} below). In the next results we will not be assuming this mild extra  hypothesis, but rather require the existence of the normalized angles.
\end{remark}

As an immediate consequence of the definition of normalized angles we have a property called the \emph{local Lorentzian Toponogov property}.

\begin{theorem}\label{Toponogov}
For any comparison neighborhood $U$ in a Lorentzian pre-length space $(X,d,\ll,\leq,\tau)$ with curvature bounded below by $k$ and any triangle $(x,y,z)$ in $U$ we have
\[
\measuredangle yxz \leq \dst\frac{\wangle_{k}yxz}{ac}, \ \measuredangle xyz \leq \dst\frac{\wangle_{k}xyz}{ab}, \ \measuredangle xzy \leq \dst\frac{\wangle_{k}xzy}{bc}
\]
\end{theorem}

Once we have a notion of normalized angle in a Lorentzian length space, we show an adapted version of the hinge theorem.

\begin{theorem}\label{HingeTheorem}
Let $U$ be a comparison neighborhood of a Lorentzian pre-length space $(X,d,\ll,\leq,\tau)$ with curvature bounded below by $k$. Let us consider a timelike geodesic triangle $(x,y,z)$ in $U$, realized by maximal causal curves $\alpha$, $\beta$, $\gamma$ whose lengths satisfy timelike size bounds for $k$ and such that $\measuredangle yxz$ is finite and $\measuredangle \bar{y}\bar{x}\bar{z} = ac \measuredangle yxz$. Then, for any $s,t\in (0,a]\times (0,c]$ satisfaying $\alpha(s)\ll \gamma(t)$ or $\gamma(t)\ll \alpha(s)$, we have
\[
\tau(\alpha(s),\gamma(t)) \geq \bar{\tau}(\bar{\alpha}(s),\bar{\gamma}(t)).
\]
\end{theorem}

\begin{proof}
Let us take $(s,t)\in (0,a]\times (0,c]$ such that $\alpha(s)\ll \gamma(t)$. If $\bar{\tau}(\bar{\alpha}(s),\bar{\gamma}(t))=0$ then we are done. Otherwise, suppose $\bar{\tau}(\bar{\alpha}(s),\bar{\gamma}(t))>0$. Let $triangle x_1y_1z_1$ be a comparison triangle for triangle $(x,\alpha(s),\gamma(t))$. Applying Lemma \ref{RescalAngulo} we get
\[
\measuredangle \bar{\alpha}(s)\bar{x}\bar{\gamma}(t) = \frac{st}{ac} \measuredangle \bar{y}\bar{x}\bar{z} = \frac{st}{ac} \left( ac \measuredangle yxz \right) = st \measuredangle yxz
\]
Because of the definition of the normalized angle we have
\[
\measuredangle \bar{\alpha}(s)\bar{x}\bar{\gamma}(t) = st \measuredangle yxz 
\leq st \theta^{k}_{\alpha,\gamma}(s,t) 
= st \left(\dst\frac{\wangle_{k} \alpha(s)x\gamma(t)}{st} \right) 
= \wangle_{k} \alpha(s)x\gamma(t) = \measuredangle y_1x_1z_1,
\]
therefore $\measuredangle \bar{\alpha}(s)\bar{x}\bar{\gamma}(t) \leq  \measuredangle y_1x_1z_1$. By Lemma \ref{HingeLemmaMk} we conclude $\bar{\tau}(y_1,z_1)\geq \bar{\tau}(\bar{\alpha}(s),\bar{\gamma}(t))$. But $\tau(\alpha(s),\gamma(t)) = \bar{\tau}(y_1,z_1)$, thus
\[
\tau(\alpha(s),\gamma(t)) \geq \bar{\tau}(\bar{\alpha}(s),\bar{\gamma}(t)).
\]
\end{proof}

%%%%%%%%%%%%%%%%%%%%%%%%%
%%%%%%%%%%%%%%%%%%%%%%%%%     Straightening
%%%%%%%%%%%%%%%%%%%%%%%%

For metric length spaces of bounded curvature the following fact about the sum of adjacent angles is well known: for any geodesic segment $\gamma$ from $x$ to $y$, a point $p\in\gamma$ and a point $q\not\in\gamma$  we have
\begin{equation}\label{SumAnglesPi}
\measuredangle xpq + \measuredangle qpy = \pi
\end{equation}
%In fact, length spaces (for which all angles exists) satisfying equation \ref{SumAnglesPi} and the local version of Toponogov Theorem have curvature bounded. 
In the case of Lorentzian model spaces of constant curvature we have the next lemma. A detailed proof can be found in \cite{Kirch}.

\begin{lemma}\label{SumaAngleCero}
Let us consider a triangle $\triangle pqr$ satisfying size bounds for $k$ in $\mathbb{M}_{k}^{L}$. Let $m$ be a point on the side $\beta$ joining $p$ with $r$ 
%(so $\tilde{\beta}$ is the geodesic traveled in the opposite direction)
and suppose this geodesic $\beta$ is parametrized by $[0,1]$, so there exists $\lambda\in[0,1]$ such that $m=\beta(\lambda)$. Then
\[
(1-\lambda)\measuredangle qmp + \lambda \measuredangle rmq =0.
\]
\end{lemma}

One of the most important issues that need to be addressed in comparison geometry consists on relating the measures of the angles of the  comparison triangles coming from the subdivision of a given triangle. The following results deals with this situation.

\begin{lemma}[Straightening lemma for shoulder angles]\label{SLemma}
Let $\triangle qpr$, $\triangle q_1p_1m_1$ and $\triangle q_2m_2r_2$ be three timelike geodesic triangles satisfying curvature bounds for $k$ in $\mathbb{M}_{k}^{L}$, and $m$ a point on the side joining $p$ to $r$. Suppose $\overline{\tau}(q,p)=\overline{\tau}(q_1,p_1)$, $\overline{\tau}(q,r)= \overline{\tau}(q_2,r_2)$, $\overline{\tau}(p,m)= \overline{\tau}(p_1,m_1)$, $\overline{\tau}(m,r)= \overline{\tau}(m_2,r_2)$ and $\overline{\tau}(q_1,m_1)= \overline{\tau}(q_2,m_2)$. If
\[
\left(1 - \frac{\overline{\tau}(p,m)}{\overline{\tau}(p,r)}\right)\wangle q_1m_1p_1 + \frac{\overline{\tau}(p,m)}{\overline{\tau}(p,r)} \wangle q_2m_2r_2 \geq 0,
\]
then 
\[
\wangle qpm \geq \wangle q_1p_1m_1 \ and \ \wangle qrm \geq \wangle q_2r_2m_2.
\]
\end{lemma}

\begin{remark}
\begin{enumerate}
    \item The statement obtained by reversing all inequalities in the above Lemma holds true as well.
\item We have analogous results when the point $m$ is located in any of the remaining sides of the triangle $m$.
\end{enumerate}
\end{remark}

We end our analysis on non-normalized angles by proving a converse of the straightening lemma that will prove useful in establishing the existence of angles in 

\begin{lemma}\label{ConverseSL}
Suppose $\triangle {q}{p}{r}$ is a triangle satisfying size bounds for $K$ in a model space of curvature $K$. Let ${m}$ be a point on the side joining ${p}$ to ${r}$, and set $\lambda=\lambda_{{m}} \in[0,1]$ the affine parameter corresponding to point $m$ when the side is parametrized by $[0,1]$. Let $\triangle q_1p_1m_1$ and $\triangle q_2m_2r_2$ be triangles in respective model spaces of curvature $K$, where $\bar\tau (q_1,m_1)=\tau\vert (q_2,m_2)$, $ \overline{\tau}(q_1,p_1)=\overline{\tau}\tilde{q},\tilde{p})$, $\overline{\tau}( q_2,r_2)=\overline{\tau}({q},{r})$, $\overline{\tau}( p_1,m_1)=\overline{\tau}({p},{m})$ and $\overline{\tau}( m_2,r_2)=\overline{\tau}({m},{r})$. If
\[
\measuredangle {q}{p}{m}\geq \measuredangle q_1p_1m_1 \ \mbox{and} \ \measuredangle {q}{r}{m}\geq \measuredangle q_2r_2m_2
\]
then 
\[
(1-\lambda)\measuredangle p_1m_1q_1 + \lambda\measuredangle r_2m_2q_2\geq 0.
\]
\end{lemma}

\begin{proof}
Because $\measuredangle {q}{p}{m}\geq \measuredangle q_1p_1m_1$, then applying  Lemma \ref{HingeLemmaMk} we have $\bar\tau (q,m)\leq \bar\tau ( q_1,m_1)$. But by Lemma \ref{Hinge2}, the shoulder angles in triangles $\triangle {q}{p}{m}$ and $\triangle q_1p_1m_1$ satisfy $\measuredangle {p}{m}{q} \leq \measuredangle p_1m_1q_1$. Similarly, $\measuredangle {q}{m}{r} \leq \measuredangle q_2m_2r_2$. In conclusion
\[
(1-\lambda)\measuredangle p_1m_1q_1 + \lambda\measuredangle r_2m_2q_2 \geq (1-\lambda)\measuredangle \tilde{p}\tilde{m}\tilde{q} + \lambda \measuredangle \tilde{q}\tilde{m}\tilde{r} =0.
\]
\end{proof}

%%%%%%%%%%%%%%%%%%%%%%%%%%%%%%%%%%%%%%%%%%%
%%%%%%%%%%%%%%%%%%%%%%%%%%%%%%%%%%%%%%%%%%%  Angle existence    
%%%%%%%%%%%%%%%%%%%%%%%%%%%%%%%%%%%%%%%%%%%

We now establish the same inequality as above for Lorentzian pre-length spaces with lower curvature bounds.

\begin{proposition}\label{SumGreaterCero}
Let $(X,d,\ll, \leq,\tau)$ be a Lorentzian pre-length space with timelike curvature bounded below by $k$. For a timelike geodesic triangle $(x,y,z)$ realized by maximal causal curves $\alpha$, $\beta$, $\gamma$ whose side lengths satisfy timelike size bounds for $k$, fix a point $m$ in $\beta$ and a maximal timelike curve $\sigma$ connecting $m$ with $x$. Then
\[
\measuredangle ymx + \measuredangle xmz \geq 0.
\]
Similar inequalities hold if $m$ is on $\alpha$ or $\gamma$.
\end{proposition}

\begin{proof}
 Let $U$ be a comparison neighborhood around $m$. Let the points $q$, $p$, $r\in U$ lie on sides $\beta_{my}$, $\beta_{mz}$ and $\sigma$, where the latter is a maximal geodesic connecting $m$ and $x$, such that $p\ll q$. Thus, $(p,q,r)$ is a timelike geodesic triangle satisfying curvature bounds for $k$. Furthermore, set $\tau(q,m)=s$, $\tau(m,r)=t$ and $\tau(p,m)=h$.
%\begin{figure}[ht]
%\centering{
%\def\svgwidth{170pt}
%\input{DesigualdadMenorCero1.pdf_tex}
%\caption{Timelike geodesic triangle $(x,y,z)$ in $X$}
%}
%\end{figure}
\begin{figure}[ht]
\centering{
\includegraphics[scale=.3]{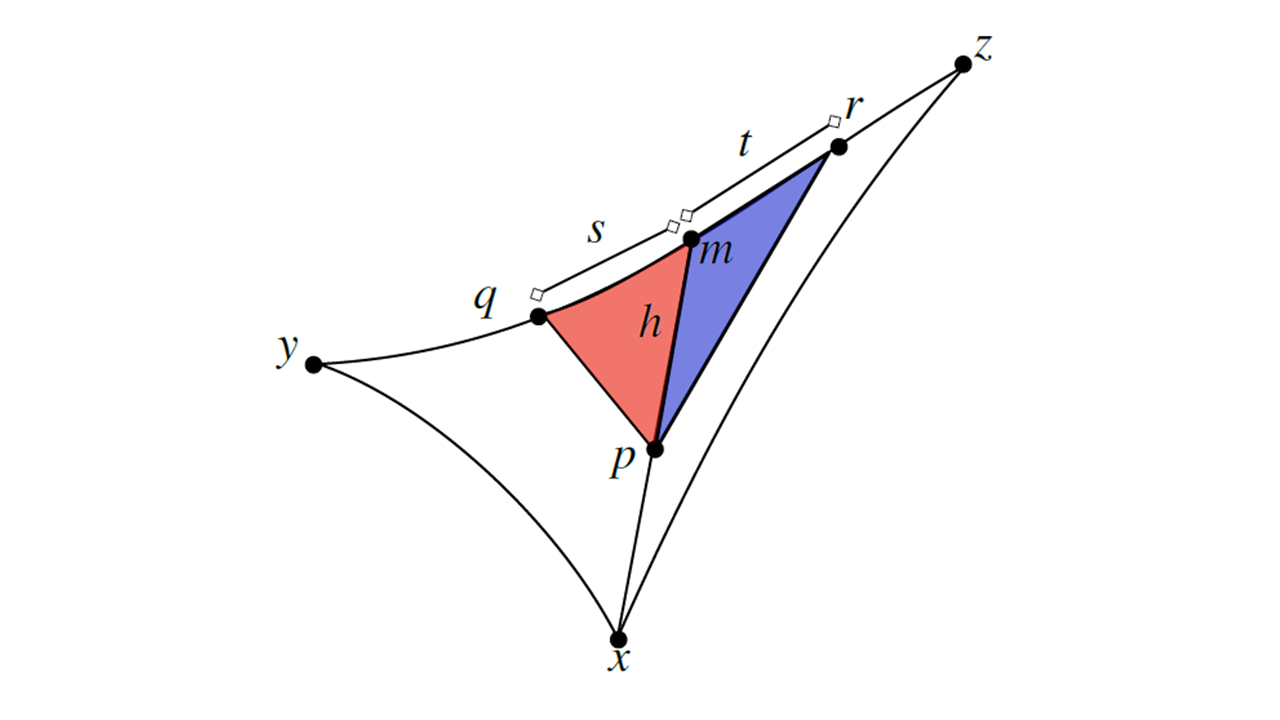}
\caption{Timelike geodesic triangle $(x,y,z)$ in $X$}
}
\end{figure}
Let $\triangle p_1q_1r_1$, $\triangle p_2q_2m_2$ and $\triangle p_3m_3r_3$ comparison triangles in $\mathbb{M}^L_k$ for triangles $(q,p,r)$, $(p,q,m)$ and $(p,m,r)$, respectively. Let us denote $\bar{\tau}(p_1,m_1)=h_1$.
%\begin{figure}[h]
%\centering{
%\def\svgwidth{200pt}
%\input{DesigualdadMenorCero2.pdf_tex}
%\caption{Comparison triangles in $\mathbb{M}_{k}^{L}$}
%}
%\end{figure}
\begin{figure}[h]
\centering{
\includegraphics[scale=.3]{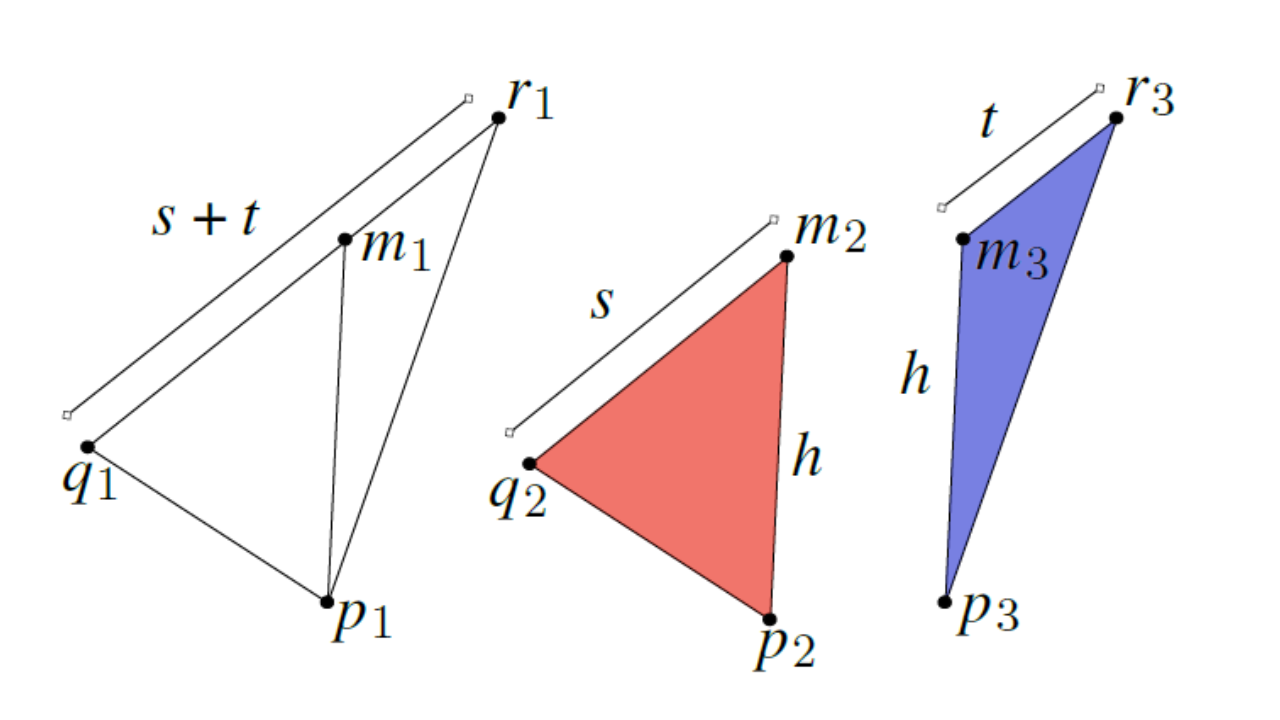}
\caption{Comparison triangles in $\mathbb{M}_{k}^{L}$}
}
\end{figure}
By the curvature condition we have $h_1\geq h$, thus by Lemma \ref{HingeLemmaMk} we have $\measuredangle p_1q_1m_1 \geq \measuredangle p_2q_2m_2$ and $\measuredangle p_1r_1m_1\geq \measuredangle p_3r_3m_3$. Now applying Lemma \ref{ConverseSL} we deduce
\[
\begin{array}{rcl}
\left(1-\frac{s}{s+t}\right)\measuredangle q_2m_2p_2 + \left(\frac{s}{s+t}\right) \measuredangle r_3m_3p_3 &\geq& 0 \\
\left(\frac{t}{s+t}\right)(sh\theta^{k}_{\sigma,\beta_{my}}(h,s)) + \left(\frac{s}{s+t}\right)(th\theta^{k}_{\beta_{mz},\sigma}(t,h)) &\geq& 0\\
\theta^{k}_{\sigma,\beta_{my}}(h,s) + \theta^{k}_{\beta_{mz},\sigma}(t,h) &\geq& 0
\end{array}
\]
Since $\angle xmz=\lim\limits_{{t,h}\to 0}\theta^{k}_{\beta_{mz},\sigma}(t,h)$ is non-negative, the above inequality provides a lower bound for $\theta^{k}_{\sigma,\beta_{my}}(h,s)$. Thus $\angle ymx $ exists also. Taking the limit when $s,t,h$ go to $0$ we conclude the desired inequality
\[
\measuredangle ymx + \measuredangle xmz \geq 0.
\]
\end{proof}

\begin{remark}
We emphasize that Proposition \ref{SumGreaterCero} is in fact a global result (as long as timelike size bounds for $k$ are satisfied). We also would like to remark that in virtue of this result, we can show that in any Lorentzian pre-length space with lower timelike curvature bounds all three normalized angles in a geodesic triangle satisfying timelike size bounds exist and are finite, provided that its sides can be extended past their vertices as maximal timelike curves.
\end{remark}

\begin{remark}
When the normalized angles exist, their value is independent of the comparison model space $\mathbb{M}_k^L$ in which the angle comparison functions $\theta^{k}_{\alpha,\gamma}$, $\theta^{k}_{\gamma,\beta}$,  $\theta^{k}_{\beta,\alpha}$ are defined. Intuitively, this is expected in the Riemannian setting since all Riemannian model spaces are conformally equivalent and angles between curves are preserved under conformal transformations. In practice, the proof relies heavily on the Law of cosines for the model spaces. The same is true for Lorentzian spaceforms, thus following closely the argument in Proposition 2.9 in \cite{Bridson} and using the Lorentzian Laws of cosines (see for example \cite{Birman}, \cite{Laws}) we can show the aforementioned independence in the context of Lorentzian pre-length spaces as well.
\end{remark}

As it turns out, the inequality of Proposition \ref{SumGreaterCero} enables us to have a partial converse to Theorem \ref{Toponogov}.

\begin{theorem}
Let $(X,d,\ll,\leq,\tau)$ be a Lorentzian pre-length space such that for any geodesic triangle $(x,y,z)$, realized by maximal timelike curves $\alpha$, $\beta$, $\gamma$ in a compatible neighborhood $U$, all normalized angles exist and satisfy
    \[
    \measuredangle yxz \leq \dst\frac{\wangle_{k}yxz}{ac}, \ \measuredangle xyz \leq \dst\frac{\wangle_{k}xyz}{ab}, \ \measuredangle xzy \leq \dst\frac{\wangle_{k}xzy}{bc}
    \]
    Furthermore,
    \begin{enumerate}
    \item If $p\in\beta$ then $\measuredangle ypx + \measuredangle xpz \geq 0$.
    \item If $p\in \alpha$ then $\measuredangle xpz + \measuredangle zpy \geq 0$.
    \item If $p\in \gamma$ with $y\ll p$ or $p\ll y$ then $\measuredangle xpy + \measuredangle ypz \geq 0$.
    \end{enumerate}
Then $(X,d,\ll,\leq,\tau)$ has timelike curvature bounded below by $k$.
\end{theorem}

\begin{proof}
 We show that $U$ satisfies the conditions of Proposition \ref{CeviansCriteria}. Here we handle the case $p\in\beta$ and since the remaining ones are completely analogous. Just to simplify notation set $b_1=\tau(y,p)$, $b_2=\tau(p,z)$ (thus $b=b_1+b_2$), $h=\tau(x,p)$ and $\bar{h}=\bar{\tau}(\bar{x},\bar{p})$. We want to prove that $\bar{h}\geq h$. Let $\triangle x_1y_1p_1$ and $\triangle x_2p_2z_2$ comparison triangles in $\mathbb{M}_{k}^{L}$ for triangles $(x,y,p)$ and $(x,p,z)$, respectively. Thus
\[
\begin{array}{rcl}
\left(1-\frac{b_1}{b}\right)\measuredangle y_1p_1x_1 + \left(\frac{b_1}{b}\right)\measuredangle x_2p_2y_2 &=& \left(\frac{b_2}{b}\right)\measuredangle y_1p_1x_1 + \left(\frac{b_1}{b}\right)\measuredangle x_2p_2y_2 \\
&\geq & \left(\frac{b_2}{b}\right) \left(b_1h\measuredangle ypx\right) + \left(\frac{b_1}{b}\right) \left(b_2h\measuredangle xpz\right) \\
&=& \left(\frac{b_1b_2h}{b}\right) \left(\measuredangle ypx + \measuredangle xpz \right) \geq 0, \\
\end{array}
\]
therefore $\left(1-\frac{b_1}{b}\right)\measuredangle y_1p_1x_1 + \left(\frac{b_1}{b}\right)\measuredangle x_2p_2y_2 \geq 0$. Then by Lemma \ref{SLemma} we get $\measuredangle \bar{x}\bar{y}\bar{p} \geq \measuredangle x_1y_1p_1$. Using Lemma \ref{HingeLemmaMk} we conclude $\bar{\tau}(\bar{x},\bar{p})\geq \bar{\tau}(x_1,p_1)$, which implies $\bar{h}\geq h$.
\end{proof}

\section{First variation for nonnegatively curved Lorentzian length spaces}\label{sec:firstvar}

This section is devoted to proving a first variation formula for globally hyperbolic length spaces with timelike curvature bounded below by zero. We first show a local version for pre-length spaces and then use it for the global result.  Throughout this section, we use repeatedily the Law of Cosines in Minkowski space $\mathbb{R}^2_1$, thus we state it for ease of reference.

\begin{lemma}[Law of Cosines for $\mathbb{R}^{2}_{1}$]\label{MinkowskiLawCosines}
Given a timelike geodesic triangle $\triangle xyz$ satisfying curvature bounds in $\mathbb{R}^{2}_{1}$ with $x\ll y \ll z$ and $\tau_{\mathbb{R}^{2}_{1}}$ the time separation function in $\mathbb{R}^{2}_{1}$ we have
\[
\begin{array}{rcl}
\tau_{\mathbb{R}^{2}_{1}}(y,z)^{2} &=& \tau_{\mathbb{R}^{2}_{1}}(x,y)^2 + \tau_{\mathbb{R}^{2}_{1}}(x,z)^2 + 2\measuredangle yxz, \\
\tau_{\mathbb{R}^{2}_{1}}(x,y)^{2} &=& \tau_{\mathbb{R}^{2}_{1}}(x,z)^2 + \tau_{\mathbb{R}^{2}_{1}}(y,z)^2 + 2\measuredangle xzy, \\
\tau_{\mathbb{R}^{2}_{1}}(x,z)^{2} &=& \tau_{\mathbb{R}^{2}_{1}}(x,y)^2 + \tau_{\mathbb{R}^{2}_{1}}(y,z)^2 + 2\measuredangle zyx, \\
\end{array}
\]
\end{lemma}

As an immediate consequence we have the following result

\begin{lemma}\label{ConvergenTriangleMinkowski}
Let $\{\triangle x_iy_iz_i\}_{i=1}^{\infty}$ a sequence of timelike geodesic triangles in $\mathbb{R}^2_1=\mathbb{M}^L_0$ such that
\begin{enumerate}
\item $\tau_{\mathbb{R}^{2}_{1}}(x_i,y_i)=a$ and $\tau_{\mathbb{R}^{2}_{1}}(x_i,z_i)=c$.
\item 
$\dst\lim_{i\to\infty} x_i = x, \dst\lim_{i\to\infty} y_i = y, \dst\lim_{i\to\infty} z_i = z$,
\item  $\lim\limits_{i\to\infty}\tau_{\mathbb{R}^{2}_{1}}(y_i,z_i)=\tau_{\mathbb{R}^{2}_{1}}(y,z)$.
\end{enumerate}
Then $\dst\lim_{i\to\infty} \measuredangle y_ix_iz_i = \measuredangle yxz$ .
\end{lemma}

\begin{proposition}[Semi-continuity of angles]\label{SemicontinuityAngles}
Let $(X,d,\ll,\leq,\tau,)$ a Lorentzian pre-length space with timelike curvature bounded by $k=0$ and $U$ a comparison neighborhood with respect to $\mathbb{R}^{2}_{1}$. Let $\{(x_i,y_i,z_i)\}_{i=1}^\infty$ a sequence of timelike geodesic triangles realized by maximal curves $\alpha_i$, $\beta_i$, $\gamma_i$ whose side lengths satisfy timelike size bounds for $0$. Suppose
\[
\dst\lim_{i\to\infty} x_i = x, \dst\lim_{i\to\infty} y_i = y, \dst\lim_{i\to\infty} z_i = z, \dst\lim_{i\to\infty} \alpha_i = \alpha, \dst\lim_{i\to\infty} \beta_i = \beta, \dst\lim_{i\to\infty} \gamma_i = \gamma,
\]
and $\alpha_i\to\alpha$, $\beta_i\to\beta$, $\gamma_i\to\gamma$
 uniformly where $\alpha$, $\beta$, $\gamma$ are maximal timelike curves that satisfy timelike size bounds for $k=0$. Finally, suppose all normalized angles are finite. Then
\[
\dst\limsup_{i\to\infty} \measuredangle y_ix_iz_i \leq \measuredangle yxz.
\]
\end{proposition}

\begin{proof}
Fix $\varepsilon>0$. Since the definition of $\measuredangle yxz$, there exists $\delta>0$ with the next property: for every $t,s<\delta$ such that $\alpha(s)\ll \gamma(t)$ or $\gamma(t)\ll \alpha(s)$ we have
\[
\measuredangle yxz > \theta^{k}_{\alpha,\gamma}(s,t) - \varepsilon.
\]
Take $s,t<\delta$ with such property and denote $p=\alpha(s)$, $q=\gamma(t)$, $p_i=\alpha_i(s)$ and $q_i=\gamma_i(t)$. If there exists a subsequence $\{p_{i_j},q_{i_j}\}$ such that $\tau(p_{i_j},q_{i_j})=0$, then by the continuity of $\tau$ in $U$ we would have $\tau(p,q)=0$, which is a contradiction. Thus, we can find $N\in\mb{N}$ such that $\tau(p_i,q_i)>0$ for all $i\geq N$. Let $\triangle \bar{x}_{i}\bar{p}_{i}\bar{q}_{i}$ and $\triangle \bar{x}\bar{p}\bar{q}$ comparison triangles for $(x_i,p_i,q_i)$ and $(x,p,q)$, respectively. Further $\dst\lim_{i\to\infty}\tau(p_i,q_i)=\tau(p,q)$. On the other hand observe that $\dst\lim_{i\to\infty} \theta^{k}_{\alpha_i,\gamma_i}(s,t) = \theta^{k}_{\alpha,\gamma}(s,t)$ since $\measuredangle p_ix_iq_i\to \measuredangle pxq$ as $i\to\infty$ because of Lemma \ref{ConvergenTriangleMinkowski}. So there exists $M\in\mb{N}$ such that
\[
\theta^{k}_{\alpha,\gamma}(s,t)> \theta^{k}_{\alpha_i,\gamma_i}(s,t)-\varepsilon,
\]
for all $i\geq M$. In conclusion, for all $i\geq \max\{N,M\}$ we have
\[
\measuredangle yxz \geq \theta^{k}_{\alpha,\gamma}(s,t) - \varepsilon > \theta^{k}_{\alpha_i,\gamma_i}(s,t)-2\varepsilon \geq \measuredangle y_ix_iz_i - 2\varepsilon.
\]
Hence
\[
\dst\limsup_{i\to\infty} \measuredangle y_ix_iz_i \leq \measuredangle yxz.
\]
\end{proof}

\begin{proposition}\label{LimInfVF}
Let $(X,d,\ll,\leq,\tau)$ be a Lorentzian pre-length space with timelike curvature bounded by $0$ and $U$ a comparison neighborhood with respect to $\mathbb{R}^{2}_{1}$. Set a timelike geodesic triangle $(x,y,z)$ in $U$ realized by maximal curves $\alpha$, $\beta$, $\gamma$ whose side lengths satisfy timelike size bounds for $0$. For every $0\leq t \leq \tau(x,y)$ we define $\ell(t)=\tau(\alpha(t),z)$. Then
\[
\dst\liminf_{t\to 0^{+}} \dst\frac{\ell(t)-\ell(0)}{t} \geq \measuredangle yxz.
\]
\end{proposition}

\begin{proof}
Let $c=\tau (x,z)$. We only consider the case when $s\in(0,c]$ satisfies $\alpha(t)\ll \gamma(s)$, since  the case $\gamma(s)\ll \alpha(t)$ is similar. Let $\triangle x_1p_1q_1$ and $\triangle x_2p_2q_2$ be comparison triangles for $(x,\alpha(t),z)$ and $(x,\alpha(t),\gamma(s))$.
%\begin{figure}[h]
%\centering{
%\def\svgwidth{250pt}
%\input{VariationFormula.pdf_tex}
%\caption{Variation Formula.}
%}
%\end{figure}
\begin{figure}[h]
\centering{
\includegraphics[scale=.3]{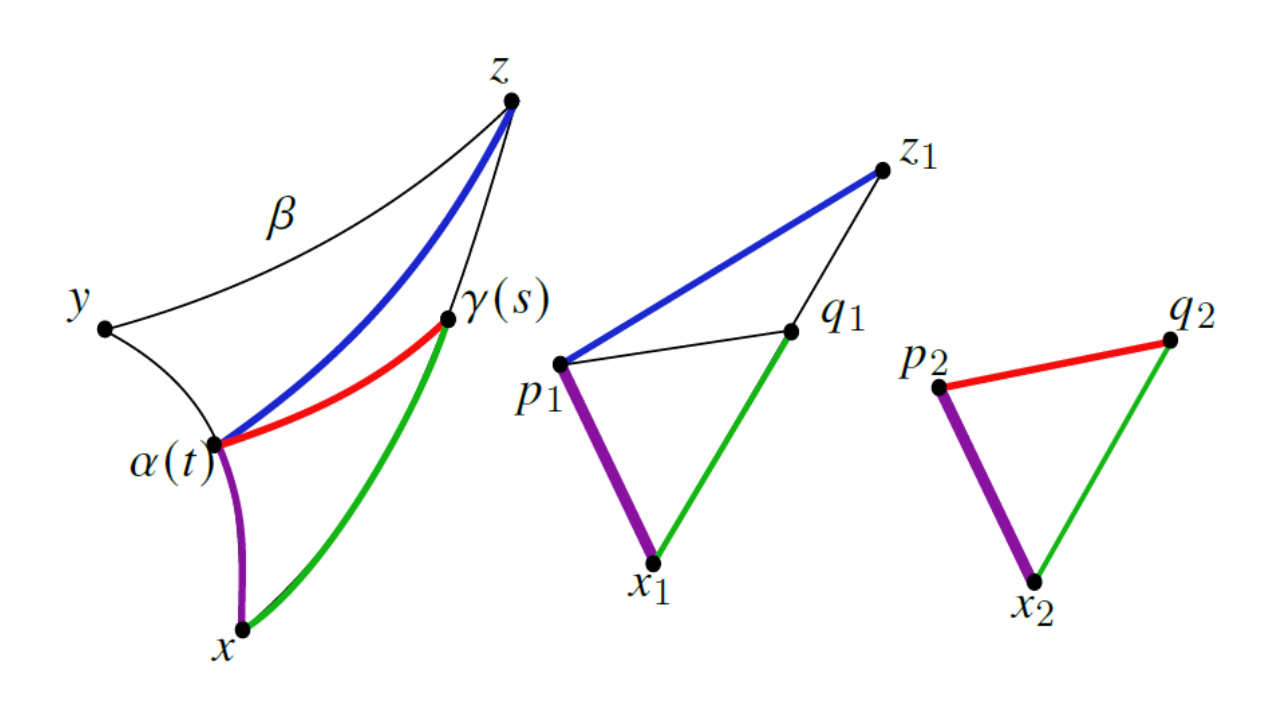}
\caption{Variation Formula.}
}
\end{figure}
Now observe that 
\[
\theta_{\alpha,\gamma}(t,s)=\dst\frac{\measuredangle p_2x_2q_2}{ts}.,
\]
and using Lemma \ref{RescalAngulo} we also have $\measuredangle p_1x_1q_1 = \frac{s}{c}\measuredangle p_1x_1z_1$. Thus, applying Lemma \ref{MinkowskiLawCosines} in triangle $\triangle p_1x_1z_1$ we get
\[
\begin{array}{rcl}
\dst\frac{\measuredangle p_1x_1q_1}{ts} &=& \dst\frac{ \frac{s}{c}\measuredangle p_1x_1z_1 }{ts} \\
&=& \dst\frac{\measuredangle p_1x_1z_1}{ct} \\
&=& \dst\frac{\ell(t)^2 - c^2-t^2}{2ct} \\
&=& \left(\dst\frac{\ell(t)-c}{t}\right) \left(\dst\frac{\ell(t)+c}{2c}\right) - \dst\frac{t}{2c},
\end{array}
\]
therefore $\dst\frac{\measuredangle p_1x_1q_1}{ts} = \left(\dst\frac{\ell(t)-c}{t}\right) \left(\dst\frac{\ell(t)+c}{2c}\right) - \dst\frac{t}{2c}$. On the other hand, by the curvature conditions we  have 
\[
\overline{\tau}(p_2,q_2) = \tau(\alpha(t), \gamma(s)) \leq \overline{\tau}(p_1,q_1),
\]
which means that $\measuredangle p_2x_2q_2\leq \measuredangle p_1x_1q_1$ because of Lemma \ref{HingeLemmaMk}. In conclusion
\[
\begin{array}{rcl}
\left(\dst\frac{\ell(t)-c}{t}\right) \left(\dst\frac{\ell(t)+c}{2c}\right) - \dst\frac{t}{2c} &=& \dst\frac{\measuredangle p_1x_1q_1}{ts} \\
&\geq& \dst\frac{\measuredangle p_2x_2q_2}{ts} \\
&=& \theta_{\alpha,\gamma}(s,t) \\
&\geq& \theta_{\alpha,\gamma}(s,t)- \dst\frac{t}{2c} \\
&\geq& \measuredangle yxz- \dst\frac{t}{2c}.
\end{array}
\]
Taking the limit when $t\to 0^{+}$ we have that $\dst\frac{t}{2c}$ goes to $0$ and the value of $\ell(t)$ tends to $c$, thus $\dst\frac{\ell(t)+c}{2c}$ approximates $1$, thus establishing the desired inequality. 
\end{proof}

\begin{remark}\label{maxangle}
Here we  have to note two important things. First, in Proposition \ref{LimInfVF} we can obtain similar inequalities for angles $\measuredangle xyz$ and $yzx$ using the corresponding distance functions $\ell$. Second, observe that angle $\measuredangle yxz$ is a function depending on the geodesics $\alpha$ and $\gamma$, but the function $\ell$ does not depend on $\gamma$. So, if $\max_{\gamma}(\measuredangle yxz)$ denotes the supremum of angles $\measuredangle yxz$ over all the timelike  maximal geodesics $\gamma$ connecting $x$ with $z$ we conclude
\[
\dst\liminf_{t\to 0^{+}} \dst\frac{\ell(t)-\ell(0)}{t} \geq \max_{\gamma}(\measuredangle yxz).
\]
\end{remark}

\begin{theorem}[First variation formula] \label{FVFLocal}
Let $(X,d,\ll,\leq,\tau)$ be a Lorentzian pre-length space with timelike curvature bounded by $0$ and $U$ a comparison neighborhood. Set a timelike geodesic triangle $(x,y,z)$ in $U$ realized by maximal curves $\alpha$, $\beta$, $\gamma$ whose side lengths satisfy timelike size bounds for $0$. For every $0\leq t \leq \tau(x,y)=a$ we define $\ell(t)=\tau(\alpha(t),z)$. Assume that a sequence of future directed causal curves $\{\gamma_i\}_{i=1}^{\infty}$ converges uniformly to $\gamma$, where $\gamma_i(0)=\alpha (t_i)$  for some sequence $\{t_i\}_{i=1}^{\infty}$, $t_i\to 0$ as $i\to \infty$. Then 
\[
\dst\lim_{i\to \infty} \dst\frac{\ell(t_i)-\ell(0)}{t_i} = \measuredangle yxz.
\]
\end{theorem}

\begin{proof}
Set $p_i=\alpha(t_i)$ and take a comparison triangle $\triangle \tilde{x}_i\tilde{p}_i\tilde{z}_i$ for $(x,p_i,z)$. Thus, by Lemma \ref{MinkowskiLawCosines} applied to triangle $\triangle \tilde{x}_i\tilde{p}_i\tilde{z}_i$ we get
\[
\measuredangle zp_ix \leq \dst\frac{\measuredangle \tilde{z}_i\tilde{p}_{i} \tilde{x}_i}{\ell(t_i)\cdot t_i} = \dst\frac{c^2-\ell(t_i)^2-t_i^{2} }{2\ell(t_i)\cdot t_i},
\]
therefore
\[
\left( \dst\frac{\ell(t_i)-c}{t_i} \right) \left( \dst\frac{\ell(t_i)+c}{2\ell(t_i)} \right) + \dst\frac{t_i}{2\ell(t_i)} \leq -\measuredangle zp_ix.
\]
Because of Proposition \ref{SumGreaterCero} we conclude 
\[
\left( \dst\frac{\ell(t_i)-c}{t_i} \right) \left( \dst\frac{\ell(t_i)+c}{2\ell(t_i)} \right) + \dst\frac{t_i}{2\ell(t_i)} \leq -\measuredangle zp_ix \leq \measuredangle yp_iz,
\]
hence
\[
\begin{array}{rcl}
\dst\limsup_{i\to\infty} \dst\frac{\ell(t)-\ell(0)}{t_i} &=& \dst\limsup_{i\to\infty} \left[\left( \dst\frac{\ell(t_i)-c}{t_i} \right) \left( \dst\frac{\ell(t_i)+c}{2\ell(t_i)} \right) + \dst\frac{t_i}{2\ell(t_i)}\right] \\
&\leq& \dst\limsup_{i\to\infty} \measuredangle yp_iz \leq \measuredangle yxz,
\end{array}
\]
this last inequality holds because of Proposition \ref{SemicontinuityAngles}. Finally, the result follows from  Proposition \ref{LimInfVF}.
\end{proof}

\begin{remark}
As in Remark \ref{maxangle} we are able to conclude that in fact
\[
\dst\lim_{i\to \infty} \dst\frac{\ell(t_i)-\ell(0)}{t_i} = \max_{\gamma}{\measuredangle yxz},
\]
because 
\[
\dst\limsup_{i\to\infty} \dst\frac{\ell(t)-\ell(0)}{t_i} \leq \measuredangle yxz \leq \max_{\gamma}(\measuredangle yxz).
\]
\end{remark}

We emphasize the first variation formula holds for timelike geodesic triangles of arbitrary size. Recall that in globally hyperbolic length spaces, the time separation $\tau$ is continuous and any two causally related points can be joined by a maximal causal curve.

\begin{theorem}\label{FVFGlobal}
Let $(X,d,\ll ,\leq, \tau)$ be a globally hyperbolic Lorentzian length space with timelike curvature bounded below by $0$. Set a timelike geodesic triangle $(x,y,z)$ in $X$ realized by maximal curves $\alpha$, $\beta$, $\gamma$ whose side lengths satisfy timelike size bounds for $0$. For every $0\leq t \leq \tau(x,y)=a$ we define $\ell(t)=\tau(\alpha(t),z)$. Assume that a sequence of future directed causal curves $\{\gamma_i\}_{i=1}^{\infty}$ converges uniformly to $\gamma$, where $\gamma_i(0)=\alpha (t_i)$ for some sequence $\{t_i\}_{i=1}^{\infty}$, $t_i\to 0$ as $i\to \infty$. Then 
\[
\dst\lim_{i\to \infty} \dst\frac{\ell(t_i)-\ell(0)}{t_i} = \measuredangle yxz.
\]
\end{theorem}

\begin{proof}
The proof is divided in two parts. First, let us take $U$ a comparison neighborhood with respect to $\mathbb{R}^{2}_{1}$ around $x$. Fix to points $y_0\in\alpha $ and $z_0\in\gamma$ in $U$ such that $y_0\ll z_0$ and define $\ell_{0}(t)=\tau(\alpha(t),z_0)$ for $0\leq t \leq \tau(x,y_0)$. Then by Theorem \ref{FVFLocal} we have
\[
\dst\lim_{i\to\infty} \dst\frac{\ell_0(t_i)-\ell_0(0)}{t_i} = \measuredangle y_0xz_0 = \measuredangle yxz.
\]
Now observe that $\ell(t_i)\geq \ell_0(t_i)+\tau(z_0,z)$, then
\[
\dst\frac{\ell(t_i)-\tau(x,z)}{t_i} \geq \dst\frac{\ell_0(t_i)-(\tau(x,z)-\tau(z_0,z))}{t_i},
\]
and therefore
\[
\dst\lim_{i\to\infty} \dst\frac{\ell(t_i)-\ell(0)}{t_i}\geq 
\dst\lim_{i\to\infty} \dst\frac{\ell_0(t_i)-\ell_0(0)}{t_i} = \measuredangle yxz.
\]
For the second part and using de continuity of $\tau$ with respect to $d$ take a large enough $i$ such that $\alpha(t_i)\in U$ and $z_i\in \gamma(i)$ with $\tau(\alpha(t_i),z_i)=r$ for a fix $r>0$. Let $(\bar{x}_i,\bar{p}_i,\bar{z}_i)$ a comparison triangle in $\mathbb{R}^{2}_{1}$ for triangle $(x,\alpha(t_i),z_i)$. Set $c_i=\tau(x,z_i)$, then by curvature conditions in $U$ and applying The Law of Cosines in $\mathbb{R}^{2}_{1}$ we have
\[
\measuredangle z_i\alpha(t_i)x \leq \dst\frac{\measuredangle \bar{x}_i\bar{p}_i\bar{z}_i}{rt_i} = \dst\frac{c_i^2-r^2-t_i^2}{2rt_i},
\]
which is equivalent to
\[
\left(\dst\frac{r-c_i}{t_i}\right) \left(\dst\frac{r+c_i}{2r}\right) + \dst\frac{t_i}{2r} \leq -\measuredangle z_i\alpha(t_i)x \leq \measuredangle y\alpha(t_i)z_i = \measuredangle y\alpha(t_i)z,
\]
this last inequalities on behalf of Proposition \ref{SumGreaterCero}. On the other hand, since $x\ll \alpha(t_i)\ll z_i \ll z$ we obtain $(\ell(t_i)-r)+ c_i\leq c$ and $r\leq c_i$, thus
\[
\left(\dst\frac{\ell(t_i)-c}{t_i}\right)\left(\dst\frac{r+r}{2r}\right) + \dst\frac{t_i}{2r} \leq \left(\dst\frac{r-c_i}{t_i}\right)\left(\dst\frac{r+c_i}{2r}\right) + \dst\frac{t_i}{2r} \leq \measuredangle y\alpha(t_i)z.
\]
When $t_i\to 0$ and using Proposition \ref{SemicontinuityAngles} we have
\[
\dst\lim\sup_{i\to \infty} \dst\frac{\ell(t_i)-c}{t_i} \leq \dst\lim\sup_{i\to \infty} \measuredangle y\alpha(t_i)z \leq \measuredangle yxz,
\]
and we are done.
\end{proof}

\begin{remark}
Here we have the same situations as in Remarks \ref{maxangle}, namely
\[
\dst\lim_{i\to \infty} \dst\frac{\ell(t_i)-\ell(0)}{t_i} = \max_{\gamma}(\measuredangle yxz).
\]
\end{remark}

\section*{Acknowledgements}
W. Barrera was partially supported by Conacyt under grants SNI 45382 and Ciencia de Frontera 21100. D. Solis was partially supported by Conacyt under grant SNI 38368. The authors are very thankful to T. Beran for insightful comments on an earlier version of this work. The authors are very thankful to the organizers of \emph{SCRI21, a tribute to Roger Penrose} for this outstanding event.

%\section*{Declarations}

%\begin{itemize}
%\item \textbf{Funding} W. Barrera was partially supported by Conacyt under grants SNI 45382 and Ciencia de Frontera 21100. D. Solis was partially supported by Conacyt under grant SNI 38368.
%\item \textbf{Conflict of interest/Competing interests} The author declare they have no conflict of interests.  
%\item Ethics approval 
%\item Consent to participate
%\item Consent for publication
%\item \textbf{Availability of data and materials} No data was collected nor used in this work.
%\item Code availability 
%\item Authors' contributions
%\end{itemize}

%%%%%%%%%%%%%%%%%%%%
%%%%%%%%%%%%%%%%%%%%  References
%%%%%%%%%%%%%%%%%%%%

\vspace{1cm}

\noindent \textbf{Waldemar Barrera}. Facultad de Matem\'aticas, Universidad Aut\'onoma de Yucat\'an, Perif\'erico Norte Tablaje 13615,  C.P. 97110, M\'erida, M\'exico. \\
bvargas@correo.uady.mx

\vspace{.3cm}

%( \Letter \ )

\noindent \textbf{Didier A. Solis}.  Facultad de Matem\'aticas, Universidad Aut\'onoma de Yucat\'an, Perif\'erico Norte Tablaje 13615,  C.P. 97110, M\'erida, M\'exico. \\
didier.solis@correo.uady.mx

\vspace{.3cm}

\noindent \textbf{Luis M. Montes de Oca}. Facultad de Matem\'aticas, Universidad Aut\'onoma de Yucat\'an, Perif\'erico Norte Tablaje 13615,  C.P. 97110, M\'erida, M\'exico. \\
mauricio.montesdeoca@alumnos.uady.mx

\end{document}